\documentclass[11pt]{amsart}
\usepackage[utf8]{inputenc}
\usepackage{amssymb}
\usepackage{amsmath}
\usepackage{mathtools}
\usepackage{tabulary}
\usepackage{booktabs}
\usepackage{setspace}
\usepackage{environ}
\usepackage{mathrsfs}
\usepackage{amsfonts}
\usepackage{pict2e}
\usepackage{enumitem}
\usepackage{tikz}
\usepackage{tikz-cd}
\usepackage[english]{babel}
\usepackage{amsthm}
\usepackage[numbers]{natbib}
\usepackage[justification=centering]{caption}
\usepackage{mathpazo}

\usepackage{geometry}
\usepackage{graphicx}
\NewDocumentCommand{\mgn}{}{\mathcal M_{g,n}}
\NewDocumentCommand{\tgn}{}{\mathcal T_{g,n}}
\NewDocumentCommand{\mgnb}{}{\overline{\mathcal M}_{g,n}}
\NewDocumentCommand{\teich}{}{\text{Teichm\"uller}}
\DeclarePairedDelimiter\norm{\lVert}{\rVert}

\DeclareMathOperator{\arcsinh}{arcsinh}
\DeclareMathOperator{\arccosh}{arccosh}
\DeclareMathOperator{\crit}{Crit}

\DeclareMathOperator{\ind}{ind}
\DeclareMathOperator{\sys}{sys}
\DeclareMathOperator{\ex}{exp}
\DeclareMathOperator{\proj}{proj}

\DeclareMathOperator{\spn}{Span}

\DeclareMathOperator{\intr}{int}
\DeclareMathOperator{\rank}{rank}
\DeclareMathOperator{\syst}{sys_T}
\DeclareMathOperator{\secsys}{sec.sys}



\newtheorem{theorem}{Theorem}[section]
\newtheorem*{main}{Main Theorem}
\newtheorem{lemma}[theorem]{Lemma}

\newtheorem{corollary}[theorem]{Corollary}
\newtheorem*{convention}{Convention}

\theoremstyle{definition}
\newtheorem{definition}[theorem]{Definition}
\newtheorem*{definition*}{Definition}

\theoremstyle{remark}
\newtheorem{remark}[theorem]{Remark}

\numberwithin{equation}{subsection}

\title{Morse theory on moduli of curves}
\author{Changjie Chen}

\begin{document}

\begin{abstract}
    We provide a new approach to studying the moduli space of curves via Morse theory and hyperbolic geometry, by introducing a family of Morse functions on the moduli space $\mgnb$ of stable curves of genus $g$ with $n$ marked point, from the $\teich$ theoretic perspective. They are weighted exponential averages of the lengths of all simple closed geodesics. These Morse functions behave well with respect to the Deligne-Mumford stratification of $\mgnb$. The critical points can be characterized by a combinatorial property named eutacticity, and the Morse index can be computed accordingly. Also, the Weil-Petersson gradient flow of the Morse functions is well defined on $\mgnb$, which can be used to build the Morse theory.
    
    These functions might be the first explicit examples of Morse functions on $\overline{\mathcal M}_{g,n}$, and the Morse handle decomposition gives rise to the first example of a natural cell decomposition of $\mgnb$ in theory, that works for all pairs $(g,n)$.
\end{abstract}

\maketitle

\tableofcontents

\section{Introduction}

The moduli space $\mgn$ of smooth projective curves, and $\mgnb$ of stable curves, of genus $g$ with $n$ marked points, are fundamental objects in algebraic geometry, topology, and many other areas. They can also be defined using other structures, such as Riemann surface structures, complete hyperbolic metrics, conformal structures, etc, due to equivalences in real dimension 2, to parametrize isomorphism classes of the respective structures.

The spaces, as varieties, stacks, topological spaces, or orbifolds, have been extensively studied with various powerful methods developed, particularly within the frameworks of algebraic geometry and topology, leading to significant results concerning the topology of the moduli space and related, although many aspects of the topology and geometry still remain subtle.

The goal of this paper is to provide a brand new perspective to study the moduli space with the use of Morse theory with differential geometry. Morse theory has been a powerful tool for understanding complicated spaces through handle decompositions, providing deep insights into their topological and geometric properties; cf. \cite{milnor2016morse}.

In this paper, we will construct and study a family of Morse functions (see definition on the next page) on the moduli space $\mgnb$, from the $\teich$ theoretic perspective, and explore how Morse theory is applied to these functions. The work is inspired by results on the \textit{systole} function, which satisfies a weak version of Morse property but unfortunately does not generate a Morse theory (cf. Remark \ref{nomorse}), although the math is mutually independent.

Geodesic-length functions play a fundamental role here. The \textit{systole} function, $\sys\colon \mgn\to\mathbb R_+$, is defined to be the length of a shortest geodesic, or equivalently the minimum of all geodesic-length functions. \textit{Topological Morse property} (cf. \cite{morse1959topologically}) was shown by Schmutz Schaller in \cite{schaller1998geometry} for the systole function on $\tgn$ with $n>0$. In \cite{akrout2003singularites}, Akrout showed that for any $\tgn$ by developing a theory on \textit{(semi-)eutacticity} differentiating surfaces by the position of gradient vectors of geodesic-length functions.

Aside from the math originating from geodesic-length functions, as another candidate, $-\log\det^*\Delta_m$ on the space of hyperbolic structures $\{m\}$ on a given surface was conjectured by Sarnak to be a Morse function, where $\det^*\Delta_m$ is the zeta regularized product of all non-zero eigenvalues of $\Delta_m$, see \cite{quine1997extremal}. It seems that very little study followed and this conjecture still remains wide open.

\begin{definition*}
Let $0<T<1$ be a parameter. For a hyperbolic surface $X\in\mgn$, let
$$\syst(X)=-T\log\left(\sum_{\gamma \text{ s.c.g. on } X} e^{-\frac1Tl_\gamma(X)}\right),$$
where s.c.g. stands for simple closed geodesic. For a nodal surface $X\in\partial\mgn$ with $s$ nodes, let $$\syst(X)=-T\log\left(s+\sum_{\gamma \text{ s.c.g. on } X} e^{-\frac1Tl_\gamma(X)}\right).$$
\end{definition*}

The $\syst$ functions are a family of weighted exponential averages of geodesic-length functions, which involve not only the shortest closed geodesics but all simple closed geodesics in a global way. As $T$ decreases to 0, $\syst$ increases and converges to $\sys$ on $\mgnb$.

\begin{convention}
    For the following remarks, throughout the paper, unless otherwise specified, a function is $C^2$-Morse whenever it is said Morse.
\end{convention}

\begin{remark}
    A Morse function is defined to be smooth, in currently the most popular version. However, to define critical points and nondegeneracy, to construct the handle decomposition, and to apply Morse theory, it is allowed to weaken the smoothness condition to $C^2$-continuity (or actually even weaker), and we may call such functions \textit{$C^2$-Morse functions}.
\end{remark}
\begin{remark}
    A slight modification of $\syst$ yields smooth Morse functions on $\mgnb$ with exactly the same properties as in the main theorems as well as almost anywhere else. Let $f$ be a smooth bump function on $\mathbb R$ that is equal to 1 on $[0,\sys(g,n)]$ and is supported on its 1-neighborhood, where $\sys(g,n):=\max_{X\in\mgn}\{\sys(X)\}$, then the function $$\sys_T^f(X)=-T\log\sum_{\gamma \text{ s.c.g. on } X}f(l_\gamma(X)) e^{-\frac1Tl_\gamma(X)}$$ is such an example as desired.
\end{remark}

\begin{main}[Part I]
\label{main1}
    For sufficiently small $T>0$, $\syst$ is a Morse function on the Deligne-Mumford compactification $\overline{\mathcal M}_{g,n}$, with the differential structure defined in Section \ref{extension}.
\end{main}

The $\syst$ functions behave well with respect to the stratification of $\mgnb$. In every single stratum, the critical points can be related to eutactic points (cf. Definition \ref{eutactic}).

\begin{main}[Part II]
\label{main2}
    In $\mgn$, $\syst$-critical points and eutactic points are in pairs, in each of which the two points, say $p_T$ and $p$, have Weil-Petersson distance $d_{WP}(p_T,p)<CT$, and the Morse index of $p_T$ is equal to the eutactic rank of $p$. Consequently, the former converges to the latter as $T\to0^+$.
\end{main}

The following states additivity of Morse index for nodal surfaces that are the union of smaller surfaces along nodes, which allows potential induction arguments.

\begin{main}[Part III]
\label{main3}
    Let $\mathcal S$ be a boundary stratum of $\mgnb$ that is canonically isomorphic to the product $\prod \mathcal M_i$ of moduli spaces, and $X=\coprod_{\{\text{nodes}\}} X_i\in\mathcal S$ be the decomposition with respect to the stratification such that $X_i\in\mathcal M_i$. The following are equivalent:

    (1) $X$ is a critical point of $\syst$ on $\mgnb$;
    
    (2) $X_i$ is a critical point of $\syst$ on $\mathcal M_i$, for all $i$.
    
    And in this case, $$\ind(X)=\sum\ind(X_i).$$
\end{main}

It is known that the Weil-Petersson metric blows up when it is extended onto $\mgnb$, however, it can still be used for the $\syst$-Morse theory due to the following.

\begin{main}[Part IV]
\label{main4}
    The Weil-Petersson gradient flow of $\syst$ on $\overline{\mathcal M}_{g,n}$ is well defined. Each gradient flow line stays in a single stratum and flows down from a critical point to a critical point in the same stratum or converges to a critical point in its boundary stratum.
\end{main}

The author believes that the $\syst$ functions are in fact smooth. To the best of the author's knowledge, they are the first explicit examples in literature, of Morse functions on the moduli space as well as the Deligne-Mumford compactification. With the Morse property, they can be used to study the topology of $\mgn$ or $\mgnb$. They can also be used to study the systole function itself given the index-preserving critical point attracting property as in Part II of the main theorems.

There has been discussion and construction of cell decompositions of the moduli space $\mgn$, when $n>0$, using quadratic differentials or fat graphs, for example, \cite{arbarello2011geometry} and \cite{norbury2010cell}. However, it does not appear that there is a universal construction of a cell decomposition of $\mgn$ or $\mgnb$ for all $(g,n)$, that is, no matter if $n$ is positive or equal to 0. The $\syst$-Morse handle decomposition of $\mgnb$ always gives rise to a cell decomposition.

As the Morse property of $\syst$ functions is established, many related questions can be asked, for example, about the distribution of critical points, that is partially answered in \cite{chen2023c2}, which implies certain homology conclusions as an application.

\begin{figure}[ht]
    \centering
    \includegraphics[width=9cm]{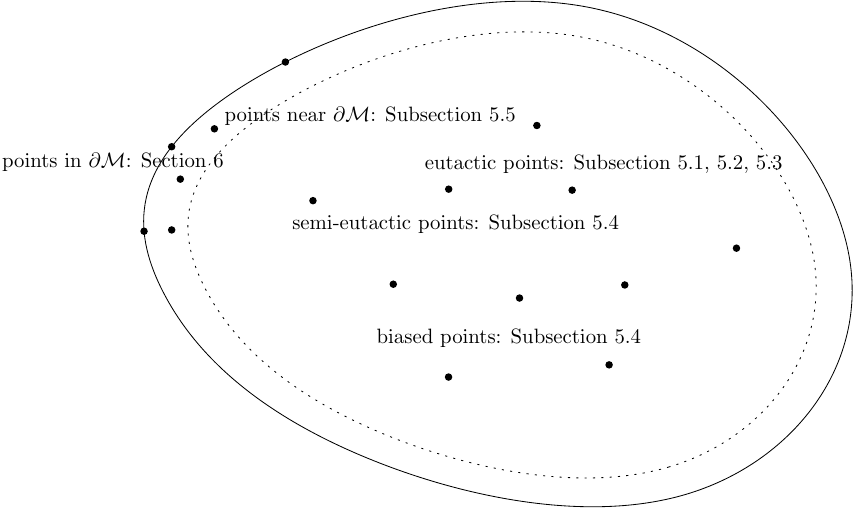}
    \caption{Organization of the paper by types of points in $\mgnb$}
    \label{fig:all types}
\end{figure}

The organization of this paper is as follows. In Section \ref{prelim}, we review some basics in hyperbolic geometry and $\teich$ theory and list a few theorems that will be used later. In Section \ref{Akrout}, we talk about Akrout's $\text{(semi-)}$eutacticity theory and his theorem on the systole function. We also develop general knowledge including the \textit{fan decomposition} that will be used as the low-level idea to study the critical points later. Subsection \ref{behaviorinmajor} is focused on the behavior of $\syst$ on the \textit{major subspace} and Subsection \ref{behavior} reveals its full local behavior around eutactic points. Subsection \ref{nondegeneracy} shows nondegeneracy of $\syst$ at its critical points; Subsection \ref{regular} studies its behavior near semi-eutactic and biased points; Subsection \ref{boundary} studies that near points near the Deligne-Mumford boundary. The extension of $\syst$ to the boundary is introduced in Section \ref{extension} to complete the proof of the main theorem. We study the Weil-Petersson gradient flow of $\syst$ and give an example of that on $\overline{\mathcal M}_{1,1}$ in Section \ref{case}. Some applications of Morse theory are presented in Section \ref{applications}.

A table of commonly used notations in this paper is attached at the end.

\section{Preliminaries and a few theorems}
\label{prelim}

\subsection{Moduli space and $\teich$ space}
Let $X$ be a smooth surface with negative Euler characteristic and $\mathcal M_{-1}(X)$ the set of all complete finite area hyperbolic metrics on $X$. Let $\textit{Diff}_+(X)$ be the set of all orientation-preserving diffeomorphisms and $\textit{Diff}_0(X)\subset\textit{Diff}_+(X)$ the subset consisting of those homotopic to the identity map. Both groups act on $\mathcal M_{-1}(X)$ by pullback of metrics. The quotient $\textit{Diff}_+(X)/\textit{Diff}_0(X)$ is known as the \textit{mapping class group} $\textit{Mod}(X)$. The \textit{moduli space} is $$\mathcal M(X):=\mathcal M_{-1}(X)/\textit{Diff}_+(X),$$ and the \textit{Teichm\"uller space} is $$\mathcal T(X):=\mathcal M_{-1}(X)/\textit{Diff}_0(X),$$ which is the universal cover of $\mathcal M(X)$ and $$\mathcal M(X)=\mathcal T(X)/\textit{Mod}(X).$$

We will frequently use a hyperbolic surface $X$ instead of a hyperbolic metric $m$ on the surface $X$, by abuse of notation, to refer to a point in the $\teich$ space or the moduli space. For surfaces $X_1$ and $X_2$ that are diffeomorphic via $f$, the moduli spaces $\mathcal M(X_1)$ and $\mathcal M(X_2)$ are isomorphic by pullback of $f$. Because of that, when $X$ is of genus $g$ and with $n$ punctures, we may use $\mathcal T_{g,n}$ and $\mgn$ for $\mathcal T(X)$ and $\mathcal M(X)$, respectively.

One may see \cite{imayoshi2012introduction} or \cite{hubbard2016teichmuller} for more on $\teich$ theory including alternative definitions using markings, Beltrami differentials, etc.

\subsection{Weil-Petersson metric}
The $L^2$-metric on $\mathcal M_{-1}$ at $m=Edx^2+2Fdxdy+Gdy^2$ is defined as
$$\left\langle \sigma_1,\sigma_2 \right\rangle_{L^2}=\int_X \frac{A_1A_2+2B_1B_2+C_1C_2}{EG-F^2}dxdy,$$
for $\sigma_i=A_idx^2+2B_idxdy+C_idy^2\in T_m\mathcal M_{-1}, i=1,2$.

The $L^2$-metric descends onto the $\teich$ space and the moduli space, in which cases is known as the \textit{Weil-Petersson metric}.

\subsection{Fenchel-Nielson coordinates}
Suppose that $X$ is a hyperbolic surface of genus $g$ with $n$ punctures. One can find exactly $3g-3+n$ mutually disjoint simple closed geodesics on $X$ so that there are no more simple closed geodesics disjoint from them. Cutting $X$ along these geodesics will yield $2g-2+n$ three times holed/punctured spheres, each of which is called a pair of \textit{pants}. A pair of pants is uniquely determined by the lengths of its three cuffs. A \textit{pants decomposition} of $X$ is any such maximal collection of disjoint simple closed geodesics.

Given a pants decomposition $\{\gamma_1,\cdots,\gamma_{3g-3+n}\}$ of $X$, for every $\gamma_i$, one can assign two numbers, $l_i\in\mathbb R_+$, known as its \textit{length parameter}, and $\tau_i\in\mathbb R$, , known as its \textit{twist parameter}. For any such $(l_i,\tau_i)_i$, one can construct a new surface by cutting $X$ along all $\gamma_i$, changing the length of each $\gamma_i$ to $l_i$ and regluing the pants along each $\gamma_i$ after twisting $\tau_i$. This surgery gives a bijective map between $(\mathbb R_+\times\mathbb R)^{3g-3+n}$ to $\mathcal T(X)$, and $(l_i,\tau_i)_i$ is known as the \textit{Fenchel-Nielson coordinates} associated to this pants decomposition.

\subsection{Geodesic-length function}
Given any non-constant non-peripheral free homotopy class $\gamma$ on $X$, there exists a unique geodesic representative with respect to a hyperbolic metric $m$, which is the unique minimizer for the length of a curve in $\gamma$. Let the \textit{geodesic-length function} $l_\gamma$ associated to $\gamma$ at $(X,m)$ be the length of the geodesic representative.

\begin{theorem}[Wolpert \cite{wolpert1987geodesic}]
\label{convex}
Any geodesic-length function $l_\gamma$ is real analytic on $\mathcal T$ and strictly convex along Weil-Petersson geodesics.
\end{theorem}

Let $c_X(L)$ be the number of closed geodesics of length at most $L$ on a finite area hyperbolic surface $X$. The following theorem gives an upper bound of the growth rate of length functions.

\begin{theorem}[Huber \cite{huber1959analytischen}]
    It is universal over all $\teich$ spaces that $$c_X(L)\sim\frac{e^L}{L}.$$
\end{theorem}

Effectively, for purposes in this paper, it would be enough if one bounds the number of simple closed geodesics of length at most $L$ on a surface, which is polynomial, see \cite{birman1985geodesics} and \cite{mirzakhani2008growth}. However, it works the same way if the greater number $c_X(L)$ is used instead. For simplicity of calculation, we suppose $c_X(L)\le ce^L$ when $L\ge 1$, for some fixed constant $c>0$.

Below we cite a few results on the Weil-Petersson gradient vector of geodesic-length functions due to Wolpert. The main tool used in the proof was the explicit Weil-Petersson geodesic length-length pairing formula given in \cite{riera2005formula}.

\begin{theorem}[Wolpert \cite{wolpert2008behavior}, \cite{wolpert2009weil}]
\label{lengthgradientshort}
Let $\gamma_i,\gamma_j$ be disjoint or identical geodesics, then $$0<\left\langle\nabla l_i,\nabla l_j\right\rangle_{\textit{WP}}-\frac{2}{\pi}l_i\delta_{ij}=O(l_i^2l_j^2),$$
where the right hand side is uniform for $l_i,l_j\le c_0$.
\end{theorem}

For the purpose of bounding the norm of vectors, we may use expressions such as $v(t)=u(t)+O(f(t))$ as $t\to t_0$, which means $\|v(t)-u(t)\|=O(f(t))$ as $t\to t_0$ when a metric or norm is specified.

\begin{theorem}[Wolpert \cite{wolpert2008behavior}, \cite{wolpert2009weil}]
\label{lengthgradientnorm}
There exists a universal constant $c>0$, such that $$\left\langle\nabla l_\gamma,\nabla l_\gamma\right\rangle_{\textit{WP}}\le c\left(l_\gamma+l_\gamma^2e^{\frac{l_\gamma}{2}}\right),$$ and $$\nabla^2l_\gamma\left(\cdot,\cdot\right)<c\left(1+l_\gamma e^{\frac{l_\gamma}{2}}\right)\left\langle\cdot,\cdot\right\rangle_{\textit{WP}},$$ for any simple closed geodesic $\gamma$, where $\nabla^2$ is the Weil-Petersson Hessian.
\end{theorem}

Let $J$ be the Weil-Petersson almost complex structure. Near the boundary of the moduli space where mutually disjoint geodesics $\beta_i$'s are short, let $\gamma_i$'s be a few geodesics disjoint from $\cup \beta_i$, such that $\{\lambda_i:=\nabla l_{\beta_i}^{\frac12},J\lambda_i,\nabla l_i\}$ is a local frame, then under this assumption there is the following.

\begin{theorem}[Wolpert \cite{wolpert2009extension}]
\label{connection}
Let $D$ be the Weil-Petersson connection, then as $l_{\beta_i}\to0$,
    \begin{align*}
        &D_{J\lambda_i}\lambda_i=\frac{3}{2\pi l_{\beta_i}^{\frac12}}J\lambda_i+O(l_{\beta_i}^{\frac32}),\\
        &D_{\lambda_j}\lambda_i=O(l_{\beta_i}^{\frac32}),\\
        &D_{J\lambda_j}\lambda_i=O(l_{\beta_i}(l_{\beta_j}^{\frac32}+l_{\beta_i}^{\frac12})), \text{when } i\neq j,\\
        &D_{\nabla l_j}\lambda_i=O(l_{\beta_i}),\\
        &D_{\lambda_j}\nabla l_i=O(l_{\beta_j}^{\frac12}),\\
        &D_{J\lambda_j}\nabla l_i=O(l_{\beta_j}^{\frac12}),\\
        &D_{\nabla l_j}\nabla l_i \text{ is continuous}.
    \end{align*}
\end{theorem}

\section{The systole function and Akrout's theorem}
\label{Akrout}

\subsection{The systole function}
The systole function \textit{sys} assigns a surface the length of a shortest (nontrivial) closed geodesic, i.e., $$\sys(X)=\min_{\gamma\text{ closed geodesic on }X}l_\gamma(X).$$ As the function respects isometries, it descends to \textit{sys}$\colon \mathcal T\to\mathbb R_+$, as well as \textit{sys}$\colon \mathcal M\to\mathbb R_+$.
\begin{definition}
For a hyperbolic surface $X$, let $S(X)$ be the set of all the shortest closed geodesics on it.
\end{definition}

For any $\gamma\in S(X)$, note that, aside from $X$ being a three times punctured sphere, $\gamma$ has to be simple (otherwise it is possible to take a nontrivial shortcut). Therefore, it is true that $$\sys(X)=\min_{\gamma\text{ simple closed geodesic on }X}l_\gamma(X).$$

It is not difficult to see that $\sys$ is continuous but not differentiable. Examples can be constructed with, for instance, Fenchel-Nielsen coordinates.

\subsection{Topological Morse function}

To show Akrout's result on the systole function, we review some definitions.

\begin{definition}[Topological Morse function, cf. \cite{morse1959topologically}]
Let $M^n$ be an $n$-dimensional manifold and $f\colon M\to\mathbb R$ a continuous function.

(1) A point $x\in M$ is called ($C^0$-)\textit{regular} if on a $C^0$-chart around $x$, $f$ is a coordinate function, otherwise it is called ($C^0$-) critical.

(2) A \textit{critical point} $x$ is nondegenerate if on a $C^0$-chart $(x^i)$ around $x$ such that $$f-f(x)=(x^1)^2+\cdots+(x^r)^2-(x^{r+1})^2-\cdots(x^n)^2.$$ In this case the \textit{index} of $f$ at $x$ is defined to be $n-r$.

(3) A continuous function is called \textit{topologically Morse} if all the critical points are nondegenerate.
\end{definition}

\subsection{Akrout's eutacticity}
\begin{definition}
Let $\{v_i\}\subset\mathbb R^n$ be a finite set of nonzero vectors. It is called \textit{eutactic} (\textit{semi-eutactic}) if the origin is contained in the interior (boundary) of the convex hull of $\{v_i\}$, with respect to the subspace topology on $\spn\{v_i\}$. We may call it \textit{biased} if it is neither eutactic nor semi-eutactic.
\end{definition}
\begin{figure}[ht]
    \centering
    \includegraphics[width=7cm]{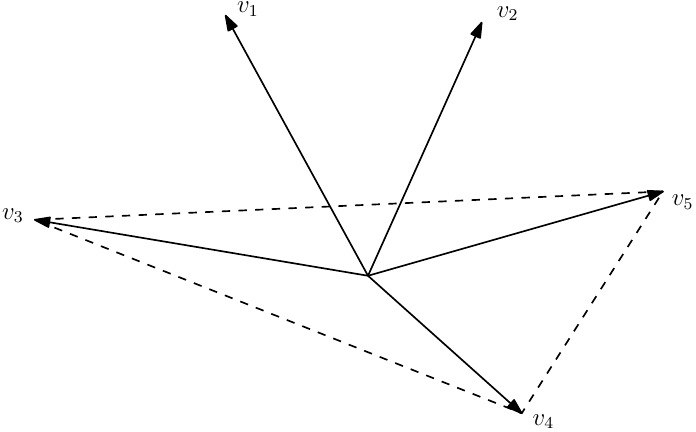}
    \caption{In this example, $\{v_1,v_2\}$ is biased, $\{v_3,v_4,v_5\}$ is eutactic, and $\{v_1,v_2,v_3,v_4,v_5\}$ is semi-eutactic}
    \label{fig:eutacticity}
\end{figure}

With a metric, we can use the following as equivalent definitions.

\begin{lemma}
Let $\{v_i\}\subset\mathbb R^n$ be a finite subset, then

$$\max_{\tau\in S^{n-1}}\min_i\{\left\langle v_i,\tau\right\rangle\}$$

is

(1) negative if and only if $\{v_i\}$ is eutactic;

(2) zero if and only if $\{v_i\}$ is semi-eutactic;

(3) positive if and only if $\{v_i\}$ is biased.
\end{lemma}

\begin{definition}
Let $S(X)=\{\gamma_1,\cdots,\gamma_r\}$ be the set of all shortest geodesics on $X$, then the gradient vectors $\nabla l_1,\cdots,\nabla l_r\in T_X\mathcal T$ are called the \textit{minimal gradients} of $X$.
\end{definition}
We extend the concept of (semi-)eutacticity to hyperbolic surfaces.
\begin{definition}
\label{eutactic}
A hyperbolic surface $X\in\mathcal T$ is called \textit{eutactic (semi-eutactic, biased)} if the set of minimal gradients $\{\nabla l_i\}_{\gamma_i\in S(X)}$ is eutactic (semi-eutactic, biased) in the tangent space $T_X\mathcal T$ (equivalently $\{dl_i\}_{\gamma_i\in S(X)}$ is eutactic (semi-eutactic, biased) in the cotangent space $T^*_X\mathcal T$ by duality).
\end{definition}

\subsection{Akrout's theorem}
In \cite{akrout2003singularites}, Akrout proved the topological Morse property for the systole function.
\begin{theorem}
The systole function is topologically Morse on $\mgn$ for any $(g,n)$. $X\in\mgn$ is a critical point if and only if $X$ is eutactic, and in this case the index is equal to $\rank\{\nabla l_\gamma\}_{\gamma\in S(X)}$.
\end{theorem}

\begin{remark}
\label{nomorse}
    The systole function can be trivially extended to $\mgnb$. However, this will lose the topological Morse property. Therefore, $\mgn$ is maximal where the systole function is topologically Morse.
    
    Given topological Morse property, it is possible to take the gradient flow locally (given any Riemannian metric), but this may not be patched globally.

    For a Morse function on a noncompact base space, the Morse differential operator can be defined as usual. However, it is not guaranteed that it yields a chain complex.

    Therefore, the systole function does not produce a Morse theory.
\end{remark}

This paper is related to, but will not use Akrout's theorem on the topological Morse property for proof.

\section{The $\syst$ functions and fan decomposition}
\label{critptsthry}
In this section, we study the convergence of the $\syst$ functions and develop a method that we may call \textit{fan decomposition}, to study the behavior of $\syst$ later.
\begin{definition}
For a hyperbolic surface $X$, let $\syst$ be a family of averages of all geodesic-length functions for simple closed geodesics on $X$, indexed by $T>0$, defined as
$$\syst(X)=-T\log\sum_{\gamma \text{ s.c.g. on } X} e^{-\frac1Tl_\gamma(X)}.$$
\end{definition}

\begin{definition}
Let $\textit{sec.sys}(X)$ be the \textit{second systole function} on a hyperbolic surface $X$, taking the value of the length of second shortest geodesics (counted without multiplicity).
\end{definition}
Note that the second systole function is upper semi-continuous and thus bounded from above on the thick part of the moduli space.
\begin{theorem}
\label{C0}
    The function $\syst$ is well defined for $T$ sufficiently small, and converges uniformly to $\sys$ as $T\to0^+$. More specifically, $$0<\sys(X)-\syst(X)<c'T,$$ where $c'$ is a constant depending only on $(g,n)$. Moreover, $\syst\colon \mathcal M\to\mathbb R$ is at least $C^2$-continuous.
\end{theorem}
\begin{proof}
    Write $$\sum_{\gamma}e^{-\frac1Tl_\gamma(X)} = \sum_{n=0}\sum_{n\le l_\gamma<n+1}e^{-\frac1Tl_\gamma(X)},$$
    and by bounding the growth rate
    \begin{align*}
        \sum_{n\le l_\gamma<n+1}e^{-\frac1Tl_\gamma(X)} &\le s_X(n+1)e^{-\frac1Tn}\\
        &\le ce^{n+1}e^{-\frac1Tn}=ce^{(1-\frac1T)n+1}.
    \end{align*}
    Therefore, $\syst$ is well defined for $T<1$, and
    \begin{align*}
        \sum_{\gamma}e^{-\frac1Tl_\gamma(X)} &= re^{-\frac1T\sys(X)}+\sum_{n=\secsys(X)}\sum_{n\le l_\gamma<n+1}e^{-\frac1Tl_\gamma(X)}\\
        &\le re^{-\frac1T\sys(X)}+\sum_{n=\secsys(X)}ce^{(1-\frac1T)n+1}\\
        &= re^{-\frac1T\sys(X)}+c\frac{e^{1+\secsys(X)}}{1-e^{1-\frac1T}}e^{-\frac1T\secsys(X)},
    \end{align*}
    where $r=\#S(X)$.
    Note that $r\ge1$ is bounded as $$\sys(g,n):=\max_{X\in\mgn}\{\sys(X)\}<\infty,$$ and so is $\secsys$. Thus,
    \begin{align*}
        |\syst(X)-\sys(X)|<c'T.
    \end{align*}
    The first and second derivatives will be studied in Lemma \ref{C1tail} and Lemma \ref{C2tail}, and that will complete the proof that $\syst$ is $C^2$-continuous.
\end{proof}
\begin{lemma}
    $\syst$ is decreasing in $T$.
\end{lemma}
\begin{proof}
Note that
    \begin{align*}
        \frac{d}{dT}\syst&=-\log\left(\sum_\gamma e^{-\frac1Tl_\gamma}\right)-\frac1T\frac{\sum_\gamma le^{-\frac1Tl_\gamma}}{\sum_\gamma e^{-\frac1Tl_\gamma}}\\
        &=-\frac{1}{\sum_\gamma e^{-\frac1Tl_\gamma}}\left(\sum_\gamma e^{-\frac1Tl_\gamma}\log\left(\sum_{\gamma'} e^{-\frac1Tl}\right)-\sum_\gamma\left(-\frac1Tl_\gamma\right)e^{-\frac1Tl_\gamma}\right)\\
        &=-\frac{1}{\sum_\gamma e^{-\frac1Tl_\gamma}}\left(\sum_\gamma e^{-\frac1Tl_\gamma}\left(\log\left(\sum_{\gamma'} e^{-\frac1Tl_{\gamma'}}\right)-\left(-\frac1Tl_\gamma\right)\right)\right)<0.
    \end{align*}
\end{proof}

Given $T$, the first derivatives can be calculated directly using the chain rule $$\nabla\syst(X)=\frac{\sum_\gamma e^{-\frac1Tl_\gamma(X)}\nabla l_\gamma(X)}{\sum_\gamma e^{-\frac1Tl_\gamma(X)}}.$$

Therefore, to study the critical points of $\syst$, it suffices to study $$\sum_\gamma e^{-\frac1T(l_\gamma(X)-\sys(X_0))}\nabla\syst(X)=\sum_\gamma e^{-\frac1T(l_\gamma(X)-\sys(X_0))}\nabla l_\gamma(X)$$ instead where $X_0$ is a point fixed.
\begin{definition}
\label{setup}
(1) Let $p$ be a critical point of the systole function, and let $$\widetilde\Omega_T(X): =\sum_\gamma e^{-\frac1T(l_\gamma(X)-\sys(p))}\nabla l_\gamma(X),$$ which is a rescaling of $\nabla\syst(X)$. We leave out $p$ in the definition of $\widetilde\Omega$ for simplicity, as there are only finitely many such critical points and the context will be clear when $\widetilde\Omega$ is used.

(2) Let $u\colon (-a,a)\to\mathcal T$ be a unit speed geodesic with $u(0)=p$ and the tangent vector $\tau:=u'(0)$. We let $$\widetilde\Omega_T(v)=\widetilde\Omega_T(t,\tau)=\widetilde\Omega_T(u(t))$$ by abuse of notation, where $X=u(t)$ and $v=t\tau$, where we omit the exponential map $\ex$ at $p$.

(3) Let $\Omega_T$ be the `main part' of $\widetilde\Omega_T$, where we replace the sum $\sum_\gamma$, over all simple closed geodesics, by $\sum_{\gamma\in S(p)}$ over all shortest geodesics. In the rest of the paper, we treat other maps like $\Phi_T$ and $\Psi_T$ in the same manner. We may also use a lowercase superscript $i$ for the $i$-th component of a sum, for example, $$\Omega^i(X)=e^{-\frac1T(l_i(X)-\sys(p))}\nabla l_i(X),$$ and an uppercase superscript $J$ for the partial sum over $J$, for example, $$\Omega^J(X)=\sum_{j\in J} e^{-\frac1T(l_j(X)-\sys(p))}\nabla l_j(X).$$

(4) Let $\pi\colon \mathbb R^n\setminus\{O\}\to S^{n-1}$ be the standard projection onto the unit sphere.

\begin{remark}
Note that if we choose an orthonormal basis for $T_p\mathcal T$, then by composing the exponential map $\ex_p$, we get a normal coordinate system locally near $p$. Therefore, we can see $\widetilde\Omega_T$ as a vector field on $\mathcal T$ near $p$ or $T_p\mathcal T$ near 0. We will recall this later. We will focus on the renormalized gradient vector $\widetilde\Omega_T$.
\end{remark}

\begin{lemma}
\label{C1tail}
When $T<1$, $$\norm*{\widetilde\Omega_T(X)-\Omega_T(X)}<c''e^{-\frac1T\secsys(X)},$$ where $c''$ is a constant depending on $(g,n)$.
\end{lemma}

\begin{proof}
We have the growth rate of the first derivative
\begin{align*}
    \frac{1}{e^{-\frac1T\syst}}\norm*{\sum_{n\le l_\gamma<n+1}e^{-\frac1Tl_\gamma(X)}\nabla l_\gamma}
    &\le \frac{1}{e^{-\frac1T(\sys+c'T)}}s_X(n+1)e^{-\frac1Tn}\norm*{\nabla l_\gamma}\\
    &\le Cce^{n+1}e^{-\frac1Tn(c(n+1+(n+1)^2e^{\frac{n+1}{2}}))^{\frac12}}\\
    &\le 2Cc^{\frac32}e^{(\frac54-\frac1T)n+\frac54}.
\end{align*}
Therefore, $\syst$ is at least $C^1$. For the approximation with $\Omega_T$, with the help of Theorem \ref{lengthgradientnorm}, we have
\begin{align*}
    \norm*{\widetilde\Omega_T(X)-\Omega_T(X)} &\le \sum_{\gamma\not\in S(X)} e^{-\frac1Tl_\gamma(X)}\norm*{\nabla l_\gamma(X)}\\
    &= \sum_{n=\secsys(X)}\sum_{n\le l_\gamma<n+1}e^{-\frac1Tl_\gamma(X)}\|\nabla l_\gamma(X)\|\\
    &\le \sum_{n=\secsys(X)}2c\sqrt{c_1}e^{n+1}e^{-\frac1Tl_\gamma(X)}e^{1+\frac{n+1}{4}}\\
    &\le 2c\sqrt{c_1} \frac{e^{\frac54+\frac94\secsys(X)}}{1-e^{\frac94-\frac1T}}e^{-\frac1T\secsys(X)}\\
    &\le c''e^{-\frac1T\secsys(X)}.
\end{align*}

The lemma follows.
\end{proof}
\end{definition}
\begin{definition}
\label{space}
Let $X$ be a point in the $\teich$ space, and $S(X)=\{\gamma_1,\cdots,\gamma_r\}$ the set of all the shortest geodesics on it. Define the subspaces of the tangent space $T_X\mathcal T$:
$$\text{the\ } \textit{major subspace}: T_X^{\sys}\mathcal T=\spn\{\nabla l_1,\cdots,\nabla l_r\},$$
$$\text{and\ the\ } \textit{minor subspace}: T_X^{\sys\perp}\mathcal T=(T_X^{\sys}\mathcal T)^\perp.$$ The notions are used mainly when $X$ is eutactic.
\end{definition}
\begin{definition}[Fan decomposition of the major subspace $T_p^{\sys}\mathcal T$]
\noindent

Let $p\in\mathcal M$ be a critical point for the systole function, and $S(p)=\{\gamma_i\}_{i\in I}$ where $I=\{1,\cdots,r\}$ is the index set, then the minimal gradient set $\{\nabla l_i\}_{i\in I}$ satisfies the eutactic condition at $p$, say $\sum_{i=1}^r a_i\nabla l_i(p)=0$ for positive $a_i$'s. Let $J$ be a subset of $I$, and we define $F_J\subset T_p^{\sys}\mathcal T$ to be the region consisting of vectors $v$ such that $J$ is the set of all indices $j$ that maximize $\left\langle \nabla l_j,v\right\rangle$. On the other hand, for any $\tau$, it assigns a multi-index $J=J(\tau)$ by inclusion $\tau\in F_J$. We extend the definition of $J(\cdot)$ to $\left(T_X^{\sys\perp}\mathcal T\right)^c$, the complement of the major subspace, by precomposing the projection, $J=J\circ\proj_{T_p^{\sys}\mathcal T}$, by abuse of notation. Let $v_J=\pi\left(\sum_{j\in J}\nabla l_j\right)$ be the unitized average.
\end{definition}
\begin{figure}[ht]
    \centering
    \includegraphics[width=7cm]{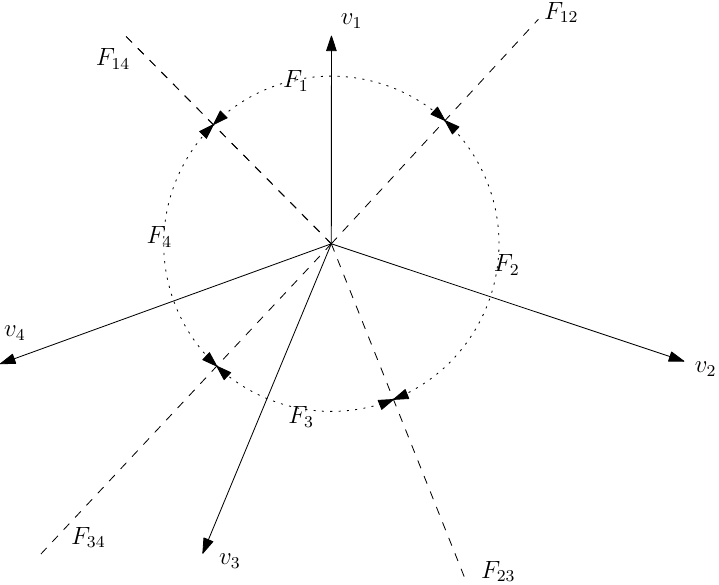}
    \caption{Fan decomposition about $\{v_1,v_2,v_3,v_4\}\subset\mathbb R^2$}
    \label{fig:fan decomposition}
\end{figure}
The fan decomposition concerns the variation of the length of the shortest geodesics at $p$ when the base point is perturbed. Fix the direction $\tau$ of a small perturbation, the geodesics indexed by $J(\tau)$ will be longer than others. If we perturb the base point in the opposite direction $-\tau$, then the geodesics indexed by $J$ will be shorter than others and are the only candidates for the shortest geodesics on $u(-t\tau)$, that is, $S(u(-t\tau))\subset\{\gamma_j\}_J$.
\begin{lemma}
\label{fan}
The following is true:

(1) $\bigcup_J F_J = T_p^{\sys}\mathcal T$.

(2) Any $F_J$ is a polygonal cone properly contained in some half space of the major subspace and $\nabla l_j\in\mathbb H(\tau)$ for any $\tau$ and all $j\in J(\tau)$, and therefore $v_J$ is well defined.

(3) $F_J\subset\mathbb H(v_J)$, equivalently, $\left\langle v_J,\tau\right\rangle>0$ if $\tau\in F_J$.

(4) There exists $D>0$ such that $|\left\langle\nabla l_j,\tau\right\rangle| > D$ for any $\tau$ unit and all $j\in J(\tau)$, i.e., $$\max_{\tau;j\in J(\tau)}\left\langle\nabla l_j,\tau\right\rangle > D \text{ and } \min_{\tau;j\in J(\tau)}\left\langle\nabla l_j,\tau\right\rangle < -D.$$

(5) The inner products $\left\langle\nabla l_j,\tau\right\rangle$ cannot be all equal. Specifically, $$\{i:\left\langle\nabla l_i,\tau\right\rangle\ge 0\}\neq\emptyset \text{ and } \{i:\left\langle\nabla l_i,\tau\right\rangle< 0\}\neq\emptyset.$$
\end{lemma}
\begin{proof}
(1) is tautological. For any $0\neq v\in T_p^{\sys}\mathcal T$, by the eutactic condition \ref{eutactic}, there exists at least one $i$ such that $\left\langle\nabla l_i,v\right\rangle>0$, otherwise, $v$ would live in the minor subspace $T_p^{\sys\perp}\mathcal T$, from which (2) and (5) follows. (3) follows directly from (2). (4) follows from the finiteness of $\{F_J\}$ and the first two propositions.
\end{proof}
\begin{remark}
If $q\in\mathcal T$ is semi-eutactic, we may take the fan decomposition of the subspace spanned by the maximal eutactic subset of all minimal gradients in a similar way, or a degenerated fan decomposition on the subspace spanned by all minimal gradients. In the latter case, we can similarly establish an inequality resembling (4) above:

There exists $D>0$ such that $$\max\left\langle\nabla l_j,\tau\right\rangle > D \text{ or } \min\left\langle\nabla l_j,\tau\right\rangle < -D$$ for any $\tau\in T^{\sys}\mathcal T$ unit.
\end{remark}
\begin{remark}
Property (4) in Lemma \ref{fan} and the remark above still hold (with possibly smaller $D$), if we allow $\tau$ to be any vector in $T_*\mathcal T$ with $\angle(\tau,T_*^{\sys\perp}\mathcal T)>\theta$ for fixed $\theta>0$.
\end{remark}

\section{Behavior of $\syst$ in $\mgn$}
We review a classic exercise in lesson one in algebraic topology.
\begin{lemma}
\label{degree}
Let $f,g\colon S^n\to S^n$ be two self maps where $f(x)+g(x)\neq0$, then $\deg f=\deg g$. Specifically, if $g=\textit{id}$, then $\deg f=1$.
\end{lemma}
\begin{proof}
$H(x,t)=\frac{(1-t)f(x)+tg(x)}{\|(1-t)f(x)+tg(x)\|}$ is a homotopy from $f$ to $g$.
\end{proof}
\subsection{In the major subspace}
\label{behaviorinmajor}
This subsection provides insight of the behavior of $\syst$ (though not complete), in a simplified setting, to help one better understand, while very little from this subsection will be directly used later.

We take the following approximation of $\Omega_T$: $$\Phi_T(t,\tau)=\sum_i e^{-\frac tT\left\langle\nabla l_i(p),\tau\right\rangle}\nabla l_i(p).$$
Note that rigorously we are considering the vector field $\ex_p^*\Omega_T$ on the tangent bundle $T(T_p\mathcal T)\cong T_p\mathcal T$, as $\ex_p$ is locally an isomorphism but we will continue to use $\Omega_T$ for simplicity by abuse of notation (we will do the same thing a few times in this paper), and the approximation $\Phi_T$ is done in the same space with the help of normal coordinates.
\begin{lemma}
\label{existence1}
Under the notation in Definition \ref{setup} and Definition \ref{space}, $\Phi_T(\cdot,\cdot)$ restricted as a map: $T_p^{\sys}\mathcal T\to T_p^{\sys}\mathcal T$ has a zero.
\end{lemma}
\begin{proof}
Consider $\Phi_T(t,-\tau)$ restricted to the $\rho T$-sphere $S^{d-1}_{\rho T}\subset T_p^{\sys}\mathcal T$, where $d=\dim T_p^{\sys}\mathcal T$ is the dimension of the major subspace and $\rho>0$ is a constant depending only on $p$ and will be determined later.

For any $\tau\in T_p^{\sys}\mathcal T$, fix a $j\in J(\tau)$, we then have
\begin{align*}
    \Phi_T(t,-\tau) &= \sum_i e^{\frac tT\left\langle\nabla l_i(p),\tau\right\rangle}\nabla l_i(p)\\
    &= e^{\frac tT\left\langle\nabla l_j(p),\tau\right\rangle}\sum_i e^{\frac tT(\left\langle\nabla l_i(p),\tau\right\rangle-\left\langle\nabla l_j(p),\tau\right\rangle)}\nabla l_i(p).
\end{align*}
Therefore, by (2) in Lemma \ref{fan}, $$\Phi_T(t,-\tau)\sim e^{\rho\left\langle\nabla l_j(p),\tau\right\rangle}\sum_{i\in J}\nabla l_i(p)\neq0,$$ when $\rho=\frac tT$ is sufficiently large and hence it induces a self map on $S^{d-1}\subset T_p^{\sys}\mathcal T$ by post-composing the projection $\pi$. It can be formulated as follows:
\begin{align*}
    \pi\circ\Phi_T(\rho,-\tau)|_{\tau\in S^{d-1}} &= \pi\circ\Phi_T(t,-\tau)|_{t\tau\in S^{d-1}_{\rho T}}\\
    &= \pi\left(e^{\rho\left\langle\nabla l_j(p),\tau\right\rangle}\sum_i e^{\rho(\left\langle\nabla l_i(p),\tau\right\rangle-\left\langle\nabla l_j(p),\tau\right\rangle)}\nabla l_i(p)\right)\\
    &\to v_J \text{\ as\ } \rho\to\infty.
\end{align*}
Therefore, there exists $\rho_J>0$ such that $$\angle\left(\Phi_T(t,-\tau), \tau\right)<\frac{\pi}{2},$$ when $t>\rho_J T$, according to (3) in Lemma \ref{fan}. By finiteness of $\{J\}$, take $\rho>\max_J\{\rho_J\}$ and the angle condition is then satisfied for all $\tau\in S^{d-1}\subset T_p^{\sys}\mathcal T$. By Lemma \ref{degree}, $\Phi_T\left(\cdot,-\cdot\right)|_{S^{d-1}_{\rho T}}$ has degree 1, and equivalently, if we do not negate $\tau$, $$\deg\left(\pi\circ\Phi_T|_{S^{d-1}_{\rho T}}\right)=(-1)^d.$$ Therefore, $\Phi_T$ has a zero in the interior of the $\rho T$-sphere.
\end{proof}

The following lemma establishes uniqueness, which still holds for $\widetilde\Omega_T$, as we will see later, while the linearity below is only for $\Phi_T$.
\begin{lemma}
\label{linearity}
The zero in the above lemma is unique and it linearly approaches $p$ in $T$.
\end{lemma}
\begin{proof}
Uniqueness: Suppose $v_1\neq v_2$ are two zeros for $\Phi_T(v)$. Consider the function $f(s):=\left\langle\Phi_T((1-s)v_1+sv_2),v_1-v_2\right\rangle$ that vanishes at $0$ and $1$. Note that $f$ is strictly increasing since each summand of $$f(s)=\sum_ie^{-\frac1T\left\langle\nabla l_i,v_1\right\rangle}e^{\frac sT\left\langle\nabla l_i,v_1-v_2\right\rangle}\left\langle\nabla l_i,v_1-v_2\right\rangle$$ is non-decreasing and at least one of them is strictly increasing, leading to a contradiction.

Linearity: Suppose $(t_0,\tau_0)$ is the zero for $\Phi_{T_0}$, then $\Phi_T\left(\frac{t_0}{T_0}T,\tau_0\right)=\Phi_{T_0}(t_0,\tau_0)=0$, that is, $\left(\frac{t_0}{T_0}T,\tau_0\right)$ is the zero for $\Phi_{T_0}$.
\end{proof}
Putting Lemma \ref{existence1} and Lemma \ref{linearity} together, we have:
\begin{theorem}
$\Phi_T$ restricted on $T_p^{\sys}\mathcal T$ has a unique zero, and it linearly approaches $p$ in $T$.
\end{theorem}

\subsection{Existence of critical points near eutactic points}
\label{behavior}
Recall that we suppose $p$ is a critical point for $\sys$, that is, $p$ is eutactic. Although the linear approximation $\Phi_T$ restricted on the major subspace $T_p^{\sys}\mathcal T$ has a unique zero, it is not strong enough to draw the same conclusion for $\nabla\syst$ when the orthogonal complement $T_p^{\sys\perp}\mathcal T$ is involved.

Note that the $t$ component in the coordinates $(t,\tau)$ that we are using for a small neighborhood of $p$ is the radial coordinate. As a reminder, we omit the exponential map `$\ex_p^*$' in $\ex_p^*l_i$, $\ex_p^*\nabla l_i$, and $\ex_p^*\Omega_T$, etc., by abuse of notation, while we use the $(t,\tau)$-coordinates. Since $l_i$ is real analytic, we can write $$l_i(t,\tau)=l_i(p) + t\left\langle\nabla l_i(p),\tau\right\rangle+\frac12t^2\nabla^2l_i(p)(\tau,\tau) + O_i(t^3)$$ and $$\nabla l_i(t,\tau)=\nabla l_i(p)+t\nabla^2l_i(p)(\tau,\cdot)+O_i(t^2).$$

We take the following second order approximation with respect to $T$ as follows for $t=\rho T$, i.e., when $t\tau\in S^{n-1}_{\rho T}$,
\begin{align*}
& e^{-\frac1T(l_i(t,\tau)-\sys(p))}\nabla l_i(t,\tau)\\
&= e^{-\frac1T(l_i(t,\tau)-\sys(p))}(\nabla l_i(p)+\rho T\nabla^2l_i(p)(\tau,\cdot)+O(T^2))\\
&= e^{\rho\left\langle\nabla l_i(p),\tau\right\rangle+\frac12\rho^2T\nabla^2l_i(p)(\tau,\tau) + O(T^2)}\nabla l_i(p)+\rho Te^{\rho\left\langle\nabla l_i(p),\tau\right\rangle + O(T)}\nabla^2l_i(p)(\tau,\cdot)\\ &+ e^{-\frac1T(l_i(t,\tau)-\sys(p))}O(T^2)\\
&= e^{\rho\left\langle\nabla l_i(p),\tau\right\rangle+\frac12\rho^2T\nabla^2l_i(p)(\tau,\tau)}\nabla l_i(p)+\rho Te^{\rho\left\langle\nabla l_i(p),\tau\right\rangle}\nabla^2l_i(p)(\tau,\cdot) + O(T^2).
\end{align*}

Therefore, if we write
\begin{align*}
    \Omega_T(t,\tau) &= \sum_ie^{-\frac 1T(t\left\langle\nabla l_i(p),\tau\right\rangle + \frac12 t^2 \nabla^2l_i(\tau,\tau))}\nabla l_i(p)\\
    &+ t\sum_ie^{-\frac 1T(t\left\langle\nabla l_i(p),\tau\right\rangle}\nabla^2l_i(p)(\tau,\cdot)+\epsilon_2(T,t,\tau),
\end{align*}
then $\epsilon_2(T,t,\tau)=O(T^2)$ when $t=\rho T$.
\begin{definition}
\label{ep2}
With the notation above, we denote the `main part' by $$\Psi_T:=\Omega_T-\epsilon_2.$$ As we observed above, $\epsilon_2(T,t,\tau)|_{t=\rho T}=O(T^2)$.
\end{definition}
We fix a constant $\theta_0$ below that will remain fixed and will be used repeatedly in the rest of the paper.
\begin{figure}[ht]
    \centering
    \includegraphics[width=10cm]{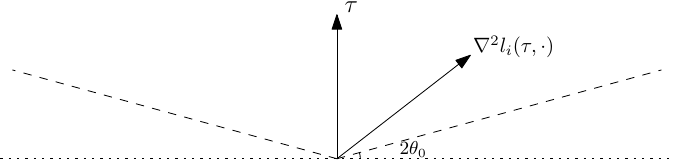}
    \caption{The angle between the two vectors is at most $\frac{\pi}{2}-2\theta_0$}
    \label{fig:theta0}
\end{figure}
\begin{definition}
Let $0<\theta_0<\frac{\pi}{6}$ be a constant (that depends on the base critical point $p$) that satisfies $$\angle(\nabla^2l_i(\tau,\cdot),\tau)<\frac{\pi}{2}-2\theta_0$$ for all $i$. The existence of $\theta_0$ is due to $\nabla^2 l_i(\tau,\tau)>0$ due to convexity in Theorem \ref{convex}, compactness of the unit sphere and finiteness of the number of shortest geodesics on $p$.
\end{definition}
Recall that $d:=\dim T_p^{\sys}\mathcal T$ is the dimension of the major subspace.
\begin{theorem}
\label{existence2}
There exists $\rho_0>0$ and $\epsilon_0=\epsilon_0(\rho)>0$ such that $\pi\circ\Psi_T(\cdot,\cdot)$ restricted on $S^{n-1}_{\rho T}\subset T_p\mathcal T$ has degree $(-1)^d$ for $\rho>\rho_0$ and $T<\epsilon_0$.
\end{theorem}
\begin{proof}
Let $i$ be the automorphism given by $$i(v_1+v_2)=v_1-v_2,$$ where $v_1\in T_X^{\sys}\mathcal T$ and $v_2\in T_X^{\sys\perp}\mathcal T$. For example, $i|_{T_X^{\sys}\mathcal T}=\textit{id}|_{T_X^{\sys}\mathcal T}$. To show the conclusion on $\pi\circ\Psi_T(\cdot,\cdot)$, we consider $i\circ\pi\circ\Psi_T(\cdot,-\cdot)$.

When $t=\rho T$, $$\Psi_T(t,-\tau)|_{S^{n-1}_{\rho T}}=\sum_ie^{\rho \left\langle\nabla l_i(p),\tau\right\rangle}\left(e^{-\frac12 \rho^2T \nabla^2l_i(p)(\tau,\tau)}\nabla l_i(p) -\rho T\nabla^2l_i(p)(\tau,\cdot)\right).$$
Let $$\Psi_1(t,-\tau)=\sum_ie^{\rho\left\langle\nabla l_i(p),\tau\right\rangle}e^{-\frac12\rho^2 T \nabla^2l_i(p)(\tau,\tau)}\nabla l_i(p)$$ and $$\Psi_2(t,-\tau)=-\sum_ie^{\rho\left\langle\nabla l_i(p),\tau\right\rangle}\rho T\nabla^2l_i(p)(\tau,\cdot).$$

We consider the following two cases by whether $\tau$ is almost perpendicular to the major subspace as shown in Figure \ref{fig:Two positions of tau}.

\begin{figure}[ht]
    \centering
    \includegraphics{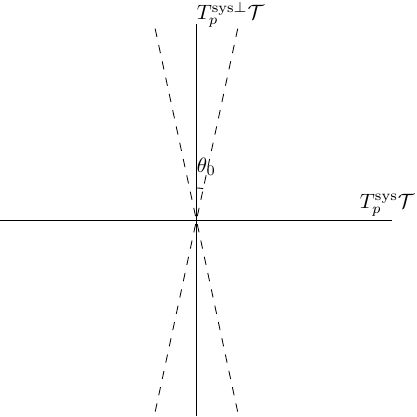}
    \caption{Two positions of $\tau$: almost perpendicular to the major subspace or not}
    \label{fig:Two positions of tau}
\end{figure}

\textbf{Case 1}: When $\angle(\tau,T_p^{\sys\perp}\mathcal T)\le\theta_0$. Let $\tau'$ be the projection of $\tau$ onto $T_p^{\sys\perp}\mathcal T$, so $\angle(\tau',\tau)\le\theta_0$. Recall that we have $\angle(\nabla^2l_i(p)(\tau,\cdot),\tau)<\frac{\pi}{2}-2\theta_0$ by the definition of $\theta_0$, and then it follows that $\angle(-\Psi_2(t,-\tau),\tau)<\frac{\pi}{2}-2\theta_0$. By triangle inequality, we have
\begin{align*}
    \angle(-\Psi_2(t,-\tau),\tau')&<\angle(-\Psi_2(t,-\tau),\tau)+\angle(\tau,\tau')\\
    &<\frac{\pi}{2}-2\theta_0+\theta_0=\frac{\pi}{2}-\theta_0,
\end{align*}
which in other words implies $i(\Psi_2(t,-\tau))\in \mathbb H(\tau')$ by definition of $i$. Note that $\Psi_1\in T_p^{\sys}\mathcal T\subset\partial\mathbb H(\tau')$, thus $$i(\Psi_T(t,-\tau))=\Psi_1(t,-\tau)+i(\Psi_2(t,-\tau))\in \mathbb H(\tau').$$ Therefore, $$\angle(i(\Psi_T(t,-\tau)),\tau)<\angle(\Psi_1+i(\Psi_2),\tau')+\angle(\tau',\tau)<\frac{\pi}{2}+\theta_0.$$

\textbf{Case 2}: When $\angle(\tau,T_p^{\sys\perp}\mathcal T)\ge\theta_0$. Similarly to what we have in (4) in Lemma \ref{fan}, by compactness of the region of such $\tau$'s in this case, there exists $D>0$ such that $\left\langle\nabla l_j(p),\tau\right\rangle > D$ for any $\tau$ and all $j\in J(\tau)$.

Let $M, m>0$ be the maximum and minimum of $\{\|\nabla l_i(p)\|,\|\nabla^2l_i(p)(\tau,\cdot)\|\}_{i,\tau}$, respectively, where positivity is given by Theorem \ref{convex}.

Let $K=\{k:\left\langle\nabla l_k,\tau\right\rangle\le0\}$ that is disjoint from $J(\tau)$. For any $j\in J(\tau)$ and $k\in K$, we have the following estimates:
\begin{align*}
    \norm*{\Psi_1^j(t,-\tau)} &> \left\langle\Psi_1^j(t,-\tau),\tau\right\rangle\\
    & = e^{\rho\left\langle\nabla l_j(p),\tau\right\rangle}e^{-\frac12\rho^2 T \nabla^2l_j(p)(\tau,\tau)}\left\langle\nabla l_j(p),\tau\right\rangle\\
    & > e^{\rho D}e^{-\frac12\rho^2 TM}D,
\end{align*}
\begin{align*}
    \norm*{\Psi_1^k(t,-\tau)} &=e^{\rho\left\langle\nabla l_k(p),\tau\right\rangle}e^{-\frac12\rho^2 T \nabla^2l_k(p)(\tau,\tau)}\norm*{\nabla l_k(p)}\\
    &\le 2e^{\frac12\rho^2 TM}M,
\end{align*}
and
\begin{align*}
    \norm*{\Psi_2(t,-\tau)} &\le \sum_ie^{\rho\left\langle\nabla l_i(p),\tau\right\rangle}\rho T\norm*{\nabla^2l_i(p)(\tau,\cdot)}\\
    &\le re^{\rho M}\rho TM,
\end{align*}
where $r=\#S(p)$.

Note that all $\left\langle\Psi_1^i(t,-\tau),\tau\right\rangle\ge0$ for $i\in(J(\tau)\cup K)^c$, then putting all terms together we have
\begin{align*}
    \left\langle\Psi_T(t,-\tau)|_{t=\rho T},\tau\right\rangle & = \left\langle\Psi_1^J+\Psi_1^K+\Psi_1^{(J\cup K)^c}+\Psi_2,\tau\right\rangle\\
    & \ge \left\langle\Psi_1^J,\tau\right\rangle-\norm*{\Psi_1^K}-\norm*{\Psi_2}\\
    & \ge e^{\rho D}e^{-\frac12\rho^2TM}D-2e^{\frac12\rho^2TM}M-re^{\rho M}\rho TM,
\end{align*}
where $J=J(\tau)$ for short.
If we set $T<e^{-\rho M}$, the expression above will go to $+\infty$ as $\rho\to\infty$. Therefore, we fix $\rho$ such that $$\left\langle\Psi_T(t,-\tau)|_{t=\rho T},\tau\right\rangle>0$$
for any $T<e^{-\rho M}$.

With the choice of $\rho$ and the two cases considered above, $$\deg i\circ\pi\circ\Psi_T(\cdot,-\cdot)=1.$$ The theorem follows.
\end{proof}

To make the approximation effective as in Theorem \ref{main}, we bound the size of $\Psi_T$.

\begin{lemma}
\label{norm2}
With the same notations above, we have the following estimates on the norms when $t=\rho T$:

(1) When $\angle(\tau,T_p^{\sys\perp}\mathcal T)\le\theta_0$, $$\norm*{\Psi_T(t,-\tau)}>\rho Te^{-\rho M\sin\theta_0}m\sin\theta_0\cos\theta_0,$$

(2) When $\angle(\tau,T_p^{\sys\perp}\mathcal T)\ge\theta_0$, $$\norm*{\Psi_T(t,-\tau)}>\frac12e^{\rho D}e^{-\frac12\rho^2TM}D.$$

Together we may claim that $\norm*{\Psi_T|_{t=\rho T}}>C(p)T.$
\end{lemma}
\begin{proof}
In the first case, let $\tau'$ be the projection of $\tau$ onto the minor subspace and consider the projection of $\Psi_T$ onto $\tau'$.

Note that $$\left\langle\nabla l_i(p),\tau\right\rangle=\|\nabla l_i(p)\|\cos\angle(\nabla l_i,\tau)>-M\sin\theta_0,$$ and as a reminder, $\theta_0$ is chosen such that $\angle(\nabla^2l_i(\tau,\cdot),\tau)<\frac{\pi}{2}-2\theta_0$, and so $$\angle(\nabla^2l_i(p)(\tau,\cdot),\tau')<\angle(\nabla^2l_i(p)(\tau,\cdot),\tau)+\angle(\tau,\tau')<\frac{\pi}{2}-\theta_0$$ and
\begin{align*}
    \left\langle\nabla^2l_i(p)(\tau,\cdot),\tau'\right\rangle&=\norm*{\nabla^2l_i(p)(\tau,\cdot)}\cdot\norm*{\tau'}\cos\angle(\nabla^2l_i(p)(\tau,\cdot),\tau')\\
    &>m\sin\theta_0\cos\theta_0.
\end{align*}
Therefore,
\begin{align*}
    \left\langle\Psi_T,-\tau'\right\rangle&=\sum_ie^{\rho\left\langle\nabla l_i(p),\tau\right\rangle}\rho T\left\langle\nabla^2l_i(p)(\tau,\cdot),\tau'\right\rangle\\
    &>\rho Te^{-\rho M\sin\theta_0}m\sin\theta_0\cos\theta_0.
\end{align*}
The first claim follows. The second follows directly from the definition of $\rho$.
\end{proof}
What Theorem \ref{existence2} tells us is that $\Psi_T$ has a zero in the interior of $B^n_{\rho T}$ by Lemma \ref{degree}. Now we are ready to prove the same conclusion for $\widetilde\Omega_T$:
\begin{theorem}[Existence]
\label{main}
With the same $\rho$ above, $\syst$ has a critical point in the $\rho T$-ball $B^n_{\rho T}\subset T_p\mathcal T$ when $T$ is sufficiently small.
\end{theorem}
\begin{proof}

To recall the idea, we take a two-step approximation $$\widetilde\Omega_T\approx \Omega_T\approx \Psi_T.$$

When $T$ is sufficiently small, any point $X$ in the geodesic ball $\ex_p(B_{\rho T})$ centered at $p$ satisfies that $\secsys(X)>\secsys(p)$, and thus by Lemma \ref{C1tail}, $$\norm*{\widetilde\Omega_T(X)-\Omega_T(X)}<c''e^{-\frac1T\secsys(X)}<c''e^{-\frac1T\secsys(p)}.$$

For the second approximation, since $\rho$ is already fixed as in Theorem \ref{existence2}, by Definition \ref{ep2}, $$\norm*{\Omega_T(X)-\Psi_T(X)}=O(T^2).$$

Therefore, the total error can be given by $$\norm*{\widetilde\Omega_T(X)-\Psi_T(X)}=O(T^2).$$

Since the angle between $\tau$ and $i\circ\Psi_T(t,-\tau)|_{S^{n-1}_{\rho T}}$ can be bounded by $\frac{\pi}{2}+\theta_0$ (which is not optimal), and $\norm*{i\circ\Psi_T(t,-\tau)|_{S^{n-1}_{\rho T}}}>C(p)T$, by the above lemma, with error taken into effect, the conditions in Lemma \ref{degree} are satisfied, and $$\deg(\widetilde\Omega_T)=\deg(\Psi_T)=(-1)^d.$$

The existence follows.
\end{proof}
\begin{corollary}
Suppose that the uniqueness (which will be proved in Theorem \ref{unique}) is satisfied for the existence above. Let $p_T$ denote the unique critical point for $\syst$ near $p$, then $d(p,p_T)$ is sublinear in $T$, meaning $d(p,p_T)<CT$ for some $C$.
\end{corollary}
We also prove the nonexistence of critical points outside of the $\rho T$-ball centered at $p$:
\begin{figure}[ht]
    \centering
    \includegraphics[width=7cm]{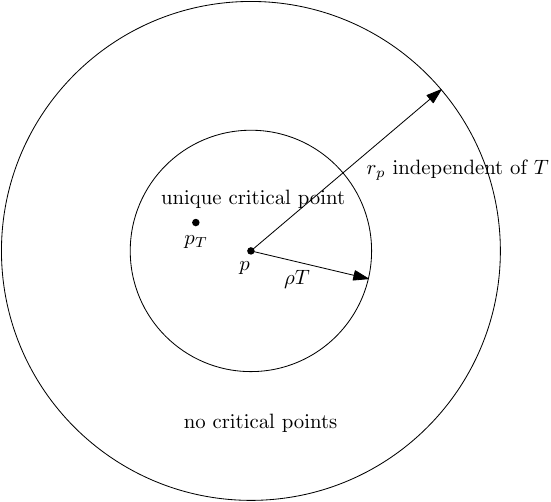}
    \caption{Near eutactic $p$: there is a unique critical point at most $\rho T$ away from $p$ in a fixed small geodesic ball independent of $T$}
    \label{fig:near eutactic p}
\end{figure}
\begin{theorem}[Nonexistence outside the $\rho T$-ball]
There exist $\rho_1>0$, $\epsilon_1=\epsilon_1(\rho)>0$ and $r_p>0$ for each $p$, such that there are no critical points of $\syst$ with distance from $p$ in $[\rho T,r_p]$ for $\rho>\rho_1$ and $T<\min\left\{\epsilon_1,\frac{r_p}{\rho}\right\}$.
\end{theorem}
\begin{proof}
We recall that $0<\theta_0<\frac{\pi}{6}$ is chosen such that for all $i$, $$\angle\left(\nabla^2l_i\left(\tau,\cdot\right),\tau\right)<\frac{\pi}{2}-2\theta_0.$$ For $\tau\not\in T_p^{\sys\perp}\mathcal T$, let $C^{\sys}(\tau,\theta)$ be the cone in the major subspace centered at $\proj_{T_p^{\sys}\mathcal T}\tau$ with angle $\theta$. Note that $v\in C^{\sys}(\tau,\theta)\times T_p^{\sys\perp}\mathcal T$ if and only if $$\angle(\proj_{T_p^{\sys}\mathcal T}v,\proj_{T_p^{\sys}\mathcal T}\tau)<\theta.$$ For example, if $\tau\in T_p^{\sys}\mathcal T$, then $C^{\sys}\left(\tau,\frac12\pi\right)\times T_p^{\sys\perp}\mathcal T=\mathbb H(\tau)$.

Let $U_p=B(p,r_p)$ be a geodesic ball centered at $p$ and $C>0$ a constant with $Cr_p<\frac12$, such that for any $q\in\overline{U}_p$, if we let $t=d(p,q)$ and $\tau=\frac{1}{t}\ex_p^{-1}(q)$, the following conditions are satisfied:

(1) By continuity, $$\left|l_i(q)-l_i(p)-t\left\langle\nabla l_i(p),\tau\right\rangle\right|<Ct^2,$$ for all $i$ and $$l_j(q)-l_j(p)<0,$$  for $j\in J(-\tau)$,

(2) By continuity, $$\|\nabla l_i(q)-\nabla l_i(p)\|<Ct,$$ for all $i$ and consequently $$\angle(\nabla l_i(q),\nabla l_i(p))<\frac12\min\left\{\theta_0,\frac12\pi-\theta_1\right\},$$ where $\theta_1$ is chosen in Case 2 below, and $$\frac12<\frac{\norm*{\nabla l_i(q)}}{\norm*{\nabla l_i(p)}}<\frac32,$$

(3) By continuity, $$\norm*{\nabla l_i(q)-\nabla l_i(p)-t\nabla^2l_i(\tau,\cdot)}<Ct^2,$$ and by taking projection onto the minor subspace, $$t\norm*{\proj_{T_p^{\sys\perp}\mathcal T}\nabla^2l_i(\tau,\cdot)}-Ct^2>0,$$

(4) By continuity, $$\max\{l_i(q)\}<\min_{\gamma\not\in S(p)}{l_\gamma(q)}.$$
$$\sys(p)+\left(r_p\|\nabla l_i\|\sin\theta_0+2Cr_p^2\right)<\min_{\gamma\not\in S(p)}{l_\gamma(q)}$$ for all $i$.

Condition (4) guarantees that we only need to consider the sum over $\gamma_i\in S(p)$ since the `tail' will be dominantly smaller. Given these constants and conditions, we again consider the following two cases:

\textbf{Case 1}: We have $\angle(\tau,T_p^{\sys\perp}\mathcal T)\le\theta_0$, then
$$\angle(\tau,\nabla l_i(p))\ge\frac{\pi}{2}-\theta_0,$$ for all $i\in I$, and therefore,

\begin{align*}
    e^{-\frac1Tl_i(q)}&>e^{-\frac1T\sys(p)}e^{-\frac1T(t\left\langle\nabla l_i(p),\tau\right\rangle+Ct^2)}\\
    &>e^{-\frac1T\sys(p)}e^{-\frac1T(t\|\nabla l_i(p)\|\sin\theta_0+Ct^2)}\\
    &>e^{-\frac1T\sys(p)}e^{-\frac1T(r_p\|\nabla l_i(p)\|\sin\theta_0+Cr_p^2)}\\
    &>e^{-\frac1T(\min_{\gamma\not\in S(p)}{l_\gamma(q)}-Cr_p^2)}.
\end{align*}
Take the projection of $\sum_ie^{-\frac1Tl_i(q)}\nabla l_i(q)$ onto the minor subspace, then we have, for a single $\nabla l_i(q)$,
\begin{align*}
    \norm*{\proj_{T_p^{\sys\perp}\mathcal T}\nabla l_i(q)}&>t\norm*{\proj_{T_p^{\sys\perp}\mathcal T}\nabla^2l_i(\tau,\cdot)}-Ct^2\\
    &\ge \rho T\norm*{\proj_{T_p^{\sys\perp}\mathcal T}\nabla^2l_i(\tau,\cdot)}-Cr_p\rho T=:C_1(p)T,
\end{align*}
and for the main part,
\begin{align*}
    &\norm*{\sum_ie^{-\frac1Tl_i(q)}\nabla l_i(q)}\\
    &\ge\norm*{\proj_{T_p^{\sys\perp}\mathcal T}\sum_ie^{-\frac1Tl_i(q)}\nabla l_i(q)}\\
    &\ge\sum_ie^{-\frac1Tl_i(q)}\norm*{\proj_{T_p^{\sys\perp}\mathcal T}(\nabla l_i(q))}\\
    &>\sum_ie^{-\frac1T\sys(p)}e^{-\frac1T(r_p\norm*{\nabla l_i(p)}\sin\theta_0+Cr_p^2)}\left(t\norm*{\proj_{T_p^{\sys\perp}\mathcal T}\nabla^2l_i(\tau,\cdot)}-Ct^2\right)\\
    &>e^{-\frac1T(\min_{\gamma\not\in S(p)}{l_\gamma(q)}-Cr_p^2)}C_1(p)T.
\end{align*}
We thus have the positivity of the norm of the gradient vector, as  the tail is dominantly smaller:
\begin{align*}
    \norm*{\sum_\gamma e^{-\frac1Tl_i(q)}\nabla l_i(q)}&\ge\norm*{\proj_{T_p^{\sys\perp}\mathcal T}\sum_ie^{-\frac1Tl_i(q)}\nabla l_i(q)}-\norm*{\sum_{\gamma\not\in S(p)} e^{-\frac1Tl_i(q)}\nabla l_i(q)}\\
    &>C_1(p)Te^{-\frac1T(\min_{\gamma\not\in S(p)}{l_\gamma(q)}-Cr_p^2)}-C_2(p)e^{-\frac1T \min_{\gamma\not\in S(p)}{l_\gamma(q)}}>0.
\end{align*}

\textbf{Case 2}: We have $\angle(\tau,T_p^{\sys\perp}\mathcal T)\ge\theta_0$.

By eutacticity, there exists $0<\theta_1<\frac12\pi$ such that for any $\tau\not\in T_p^{\sys\perp}\mathcal T$, $$C^{\sys}\left(\tau,\frac12\pi-\theta_1\right)\cap\{\nabla l_i\}\neq\emptyset.$$

Let $$J:=\left\{j:\nabla l_j(p)\in C^{\sys}\left(-\tau,\frac12\pi-\theta_1\right)\right\}$$
(note that $J$ is not short for $J(-\tau)$ in this case) and $$K:=\left\{k:\nabla l_k(p)\in C^{\sys}\left(\tau,\frac{\pi}{2}+\frac12\theta_1\right)\right\}.$$

By continuity, there exist constants $D_1<0$ and $D_2$, such that when we can make $r_p$ small enough, for any $q\in U_p=B(p,r_p)$:

(1) For $j\in J$, $$\left\langle\nabla l_j(q),\tau\right\rangle<D_1.$$

(2) For $k\in K$, $$\left\langle\nabla l_k(q),\tau\right\rangle>D_2.$$

(3) For $i\in (J\cup K)^c$, $$\left\langle\nabla l_i(q),\tau\right\rangle<0.$$

(4) $D_1<D_2$.

For any $j\in J$ and $k\in K$, we have
\begin{align*}
    \frac{\norm*{e^{-\frac1Tl_k(q)}\nabla l_k(q)}}{\norm*{e^{-\frac1Tl_j(q)}\nabla l_j(q)}}&<\frac{\frac32e^{-\frac1Tl_k(q)}\norm*{\nabla l_k(p)}}{\frac12e^{-\frac1Tl_j(q)}\norm*{\nabla l_j(p)}}\\
    &=3e^{-\frac1T(l_k(q)-l_j(q))}\frac{\norm*{\nabla l_k(p)}}{\norm*{\nabla l_j(p)}}\\
    &<3e^{\frac1T(-t(\left\langle\nabla l_k(p),\tau\right\rangle-\left\langle\nabla l_j(p),\tau\right\rangle)+2Ct^2)}\frac{\|\nabla l_k(p)\|}{\|\nabla l_j(p)\|}\\
    &<3e^{\frac1T(t(D_1-D_2)+2Ct^2)}\frac{\norm*{\nabla l_k(p)}}{\norm*{\nabla l_j(p)}}\\
    &<3e^{\rho(D_1-D_2)+2Cr_p}\frac{\norm*{\nabla l_k(p)}}{\norm*{\nabla l_j(p)}}<\epsilon,
\end{align*}
where $\epsilon<\frac1r$, if $\rho$ is sufficiently large.

We consider the projection of $\sum_ie^{-\frac1Tl_i(q)}\nabla l_i(q)$ onto the major subspace. For any $\nabla l_j(p)\in C^{\sys}\left(-\tau,\frac12\pi-\theta_1\right)$,
\begin{align*}
    -\left\langle e^{-\frac1Tl_j(q)}\nabla l_j(q),\tau\right\rangle&>-e^{-\frac1T\sys(p)}e^{-\frac1T(l_j(q)-l_j(p))}\left\langle\nabla l_j(q),\tau\right\rangle\\
    &>-D_1e^{-\frac1T\sys(p)}e^{\frac1T(-t\left\langle\nabla l_j,\tau\right\rangle-Ct^2)}\\
    &>-D_1e^{-\frac1T\sys(p)}e^{\frac1T(-D_1\rho T-Cr_p\rho T)}\\
    &=-D_1e^{-\frac1T\sys(p)}e^{\rho(-D_1-Cr_p)}.
\end{align*}

Note that for any $i\in(J\cup K)^c$, $\nabla l_i(q)$ has a negative projection onto $\tau$. Therefore,
\begin{align*}
    \norm*{\sum_Ie^{-\frac1Tl_i(q)}\nabla l_i(q)}&>-\left\langle\sum_{K^c}e^{-\frac1Tl_i(q)}\nabla l_i(q),\tau\right\rangle-\norm*{\sum_Ke^{-\frac1Tl_k(q)}\nabla l_k(q)}\\
    &\ge-\left\langle\sum_Je^{-\frac1Tl_j(q)}\nabla l_j(q),\tau\right\rangle-\norm*{\sum_Ke^{-\frac1Tl_k(q)}\nabla l_k(q)}\\
    &\ge-(1-r\epsilon)\left\langle e^{-\frac1Tl_j(q)}\nabla l_j(q),\tau\right\rangle.
\end{align*}

With the tail,
\begin{align*}
    \norm*{\sum_\gamma e^{-\frac1Tl_i(q)}\nabla l_i(q)}&>-(1-r\epsilon)\left\langle e^{-\frac1Tl_j(q)}\nabla l_j(q),\tau\right\rangle-\norm*{\sum_{\gamma\not\in S(p)} e^{-\frac1Tl_i(q)}\nabla l_i(q)}\\
    &>-(1-r\epsilon)D_1e^{-\frac1T\sys(p)}e^{\rho(-D_1-Cr_p)}-C_2(p)e^{-\frac1T \min_{\gamma\not\in S(p)}{l_\gamma(q)}}\\
    &>C_3(p)e^{-\frac1T\sys(p)}-C_2(p)e^{-\frac1T \min_{\gamma\not\in S(p)}{l_\gamma(q)}}>0.
\end{align*}
The theorem follows.
\end{proof}
\begin{remark}
Let $p_1,p_2$ be two critical points. Although it seems not feasible to calculate $r_i,i=1,2$, but we have $$d(p_1,p_2)>\max\{r_{p_1},r_{p_2}\}.$$
\end{remark}

\subsection{Nondegeneracy and local uniqueness of critical points}
\label{nondegeneracy}

With the same notations above, the Hessian $\widetilde{H}_{\syst}$ of $\syst$ on the tangent vector field $\tau=u'$ for any geodesic $u$ is given by
\begin{align*}
\widetilde{H}_T(\tau)&:=T\left(\sum_\gamma e^{-{\frac1T}l_\gamma}\right)^2\widetilde{H}_{\syst}(\tau,\tau)\\ &= \left(\sum_\gamma e^{-{\frac1T}l_\gamma}\left\langle\nabla l_\gamma,\tau\right\rangle\right)^2-\left(\sum_\gamma e^{-{\frac1T}l_\gamma}\left\langle\nabla l_\gamma,\tau\right\rangle^2\right)\left(\sum_\gamma e^{-{\frac1T}l_\gamma}\right)\\
& +T\left(\sum_\gamma e^{-{\frac1T}l_\gamma}\right)\left(\sum_\gamma e^{-{\frac1T}l_\gamma}\nabla^2l_\gamma(\tau,\tau)\right).
\end{align*}

Let $H_T$ be the `main part' of $\tilde{H}_T$, i.e., where every sum above is over $S(p)$ instead of all simple closed geodesics. By Theorem \ref{lengthgradientnorm} and a similar calculation as in Lemma \ref{C1tail} and Theorem \ref{main},
\begin{lemma}
\label{C2tail}
We have
\begin{align*}
    \left|(\widetilde{H}_T)_p-(H_T)_p\right|< c'''e^{-\frac1T(\sys(p)+\secsys(p))}.
\end{align*}
\end{lemma}
\begin{proof}
    Note that
    \begin{align*}
        \sum_\gamma e^{-{\frac1T}l_\gamma}\nabla^2l_\gamma(\tau,\tau)\le c\sum_\gamma e^{-{\frac1T}l_\gamma}(1+l_\gamma e^{\frac{l_\gamma}{2}})\le c(p)e^{-\frac1T\secsys(p)},
    \end{align*}
    then the lemma follows from Theorem \ref{C0} and Lemma \ref{C1tail}.
\end{proof}
We would like nondegeneracy of $\widetilde{H}_T$ at critical points of $\syst$, but as an observation we first have the following.
\begin{lemma}
When $T$ is sufficiently small, $(\widetilde{H}_T)_p\colon T_p\mathcal T\to\mathbb R$ is nondegenerate.
\end{lemma}
\begin{proof}
Recall that $r=\#S(X)$. Since all $\gamma_i$'s attain the same value at $p$, the above can be simplified as $$(H_T)_p(\tau)=-e^{-\frac2T\sys(p)}\left( r\sum_i\left<\nabla l_i,\tau\right>^2-\left(\sum_i\left\langle\nabla l_i,\tau\right\rangle\right)^2-Tr\sum_i\nabla l_i^2(p)(\tau,\tau)\right).$$

If $\tau\in T_p^{\sys\perp}\mathcal T$, then $$(\widetilde{H}_T)_p(\tau)>Tre^{-\frac2T\sys(p)}\sum_i\nabla^2l_i(p)(\tau,\tau)-C_pe^{-\frac1T(\sys(p)+\secsys(p))}>0.$$

If $\tau\in T_p^{\sys}\mathcal T$, by part (5) of Lemma \ref{fan} and Cauchy's inequality,
$$r\sum_i\left\langle\nabla l_i,\tau\right\rangle^2-\left(\sum_i\left\langle\nabla l_i,\tau\right\rangle\right)^2>0.$$

Similarly to part (4) in Lemma \ref{fan}, there exists $D>0$ (we may assume the same $D$ for both Lemma \ref{fan} and here for simplicity), such that $$\min_i\{\left\langle\nabla l_i(p),\tau\right\rangle\}<-D.$$

Let $$I^+=\{i:\left\langle\nabla l_i,\tau\right\rangle\ge 0\},$$ $$I^-=\{i:\left\langle\nabla l_i,\tau\right\rangle< 0\},$$ and also $r_1=\# I^+$ and $r_2=\# I^-$, then we have
\begin{align*}
& r\sum_i\left\langle\nabla l_i,\tau\right\rangle^2-\left(\sum_i\left\langle\nabla l_i,\tau\right\rangle\right)^2\\
& = r\left(\sum_{i\in I^+}\left\langle\nabla l_i,\tau\right\rangle^2 +\sum_{i\in I^-}\left\langle\nabla l_i,\tau\right\rangle^2\right) - \left(\sum_{i\in I^+}\left\langle\nabla l_i,\tau\right\rangle+\sum_{i\in I^+}\left\langle\nabla l_i,\tau\right\rangle\right)^2\\
& = \left(r_1\sum_{i\in I^+}\left\langle\nabla l_i,\tau\right\rangle^2-\left(\sum_{i\in I^+}\left\langle\nabla l_i,\tau\right\rangle\right)^2\right) + \left(r_2\sum_{i\in I^-}\left\langle\nabla l_i,\tau\right\rangle^2-\left(\sum_{i\in I^-}\left\langle\nabla l_i,\tau\right\rangle\right)^2\right)\\
& +r_2\sum_{i\in I^+}\left\langle\nabla l_i,\tau\right\rangle^2 + r_1\sum_{i\in I^-}\left\langle\nabla l_i,\tau\right\rangle^2 -2\left(\sum_{i\in I^+}\left\langle\nabla l_i,\tau\right\rangle\right)\left(\sum_{i\in I^-}\left\langle\nabla l_i,\tau\right\rangle\right)\\
& \ge r_2\sum_{i\in I^+}\left\langle\nabla l_i,\tau\right\rangle^2 + r_1\sum_{i\in I^-}\left\langle\nabla l_i,\tau\right\rangle^2 -2\left(\sum_{i\in I^+}\left\langle\nabla l_i,\tau\right\rangle\right)\left(\sum_{i\in I^-}\left\langle\nabla l_i,\tau\right\rangle\right)\\
& \ge r_2D^2+r_1D^2+2D^2\ge4D^2.
\end{align*}

Therefore, going back to $(H_T)_p$, we have
$$(H_T)_p(\tau)\le -e^{-\frac2T\sys(p)}\left( 4D^2-Tr^2M\right),$$
and
\begin{align*}
    (\widetilde{H}_T)_p(\tau) & < (H_T)_p(\tau)+C_pe^{-\frac1T(\sys(p)+\secsys(p))}\\
    &\le -e^{-\frac2T\sys(p)}\left( 4D^2-Tr^2M\right)+C_pe^{-\frac1T(\sys(p)+\secsys(p))}
\end{align*}
which is negative when $T$ is sufficiently small as $\secsys>\sys$.

Therefore, when $T$ is small enough, $H_T(\tau)<0$.
\end{proof}
Given a critical point $p_T$ of $\syst$, the Hessian can be simplified as
$$(\widetilde{H}_T)_{p_T}(\tau) = \left(\sum_\gamma e^{-{\frac1T}l_\gamma}\right)\left(\sum_\gamma e^{-{\frac1T}l_\gamma}\left(-\left\langle\nabla l_\gamma,\tau\right\rangle^2 +T\nabla^2l_\gamma(\tau,\tau)\right)\right).$$
\begin{theorem}
Let $p_T$ be a critical point (no uniqueness established yet) for $\syst$ as in Theorem \ref{main} when $T$ is sufficiently small, then $(\widetilde{H}_T)_{p_T}\colon T_{p_T}\mathcal T\to\mathbb R$ is nondegenerate.
\end{theorem}
\begin{proof}
Let $v_T=\ex_p^{-1}(p_T)\in T_p\mathcal T$, then $d(\ex_p)_{v_T}\colon T_{v_T}T_p\mathcal T\cong T_p\mathcal T\to T_{p_T}\mathcal T$ is an isomorphism when $T$ is small. By smoothness, let $C=C(p)>0$ satisfy $$\left|d(\ex_p)_{v_T}(\nabla l_i(p))-\nabla l_i(p_T)\right|<Cd(p,p_T)<C\rho T.$$

We also let $$\text{tail}=\sum_{\gamma\not\in S(p)} e^{-{\frac1T}l_\gamma}\left(-\left\langle\nabla l_\gamma,\tau\right\rangle^2 +T\nabla^2l_\gamma(\tau,\tau)\right),$$
and then $|\text{tail}|<C_pe^{-\frac1T\secsys(p)}$, similar as above.

Instead of the tangent space $T_{p_T}\mathcal T$ at $p_T$, we consider the push forward of the tangent subspaces $T_p^{\sys}\mathcal T$ and $T_p^{\sys\perp}\mathcal T$ at $p$. By continuity, results similar to part (4) of Lemma \ref{fan} and the bounds of $\{\|\nabla l_i(p_T)\|,\|\nabla^2l_i(p_T)(\tau,\cdot)\|\}_{i,\tau; T\le\epsilon}$ can be established. For simplicity, by abuse of notation, we still denote the bounds by $D,M,m$. We also assume $$l_i(p_T)<\frac12\left(\sys(p)+\secsys(p)\right)<\secsys(p),$$ for all $\gamma_i\in S(p)$ by making $T$ small enough.

When $\tau\in d\left(\ex_p\right)_{v_T}\left(T_p^{\sys}\mathcal T\right)$,
\begin{align*}
    &\left(\sum_\gamma e^{-{\frac1T}l_\gamma}\right)^{-1}\cdot (\widetilde{H}_T)_{p_T}(\tau)\\
    &\le -e^{-\frac1T\frac12(\sys(p)+\secsys(p))}(D^2-rTM) + C_pe^{-\frac1T\secsys(p)}<0,
\end{align*}
when $T$ is sufficiently small.

When $\tau\in d\left(\ex_p\right)_{v_T}\left(T_p^{\sys\perp}\mathcal T\right)$,
\begin{align*}
    &\left(\sum_\gamma e^{-{\frac1T}l_\gamma}\right)^{-1}\cdot (\widetilde{H}_T)_{p_T}(\tau)\\
    &\ge re^{-\frac1T\frac12(\sys(p)+\secsys(p))}(-(2C\rho T)^2+Tm) - C_pe^{-\frac1T\secsys(p)}>0,
\end{align*}
when $T$ is sufficiently small.

The theorem follows since $d(\ex_p)_{v_T}$ is an isomorphism.
\end{proof}
If we let the \textit{index} of a (nondegenerate) quadratic form be the number of negative eigenvalues, then we have also shown the following in the proof.
\begin{corollary}
\label{Hindex}
Let $p$ be a critical point for $\sys$, then $$\ind \widetilde{H}_{\syst}(p)=\ind \widetilde{H}_{\syst}(p_T)=(-1)^d=\ind_{\sys}(p).$$
\end{corollary}
\begin{theorem}[Uniqueness within $\rho T$-ball]
\label{unique}
When $T$ is sufficiently small, there exists a unique critical point $p_T$ within the $\rho T$-ball around $p$.
\end{theorem}
\begin{proof}
Let $p_i,i\in I$ be all the critical points in the interior of the $\rho T$-sphere described above. Consider the gradient vector field $\nabla\syst$ on the interior of $S_{\rho T}$ that is nonzero on the boundary. Note that $$\ind_{p_i}(\nabla\syst)=\ind \widetilde{H}_{\syst}(p_i).$$

Therefore, when $T$ is sufficiently small, by Theorem \ref{existence2}, we have
\begin{align*}
(-1)^d&=\deg(\nabla\syst|_{S_{\rho T}})=\sum_I\ind_{p_i}\left(\nabla\syst\right)\\
&=\sum_I\ind \widetilde{H}_{\syst}(p_i)=\#I\cdot(-1)^d,
\end{align*}
which implies $\#I=1$, i.e., the critical point near $p$ is unique.
\end{proof}
The theorem eventually validates the definition we have used above:
\begin{definition}
Suppose $T$ is sufficiently small. Let $p_T$ be the unique critical point near $p$.
\end{definition}

\subsection{Regular points in $\mgn$}
\label{regular}
By Akrout's theorem, regular points for the systole function are exactly non-eutactic points, namely, biased and semi-eutactic points. We have the following local results for $\syst$.
\begin{theorem}[Biased points]
For any biased point $q\in\mgn$, there exist a neighborhood $V=V(q)$ of $q$ and $T_0=T_0(q)>0$ such that any $q'\in V$ is regular for $\syst$ for $T<T_0$.
\end{theorem}
\begin{proof}
Let $I$ be the index set for $S(q)$. Since $q$ does not satisfy the semi-eutactic condition, there exists $\tau\in T_q^{\sys}\mathcal T$ such that $\left\langle\nabla l_i(q),\tau\right\rangle>0$ for all $i\in I$, equivalently, $$\nabla l_i(q)\in C\left(\tau,\frac12\pi-\theta_2\right),$$ for some $0<\theta_2(q)<\frac12\pi$, where $C(\tau,\theta)$ is the cone in the tangent space $T_q\mathcal T$ centered at $\tau$ with angle $\theta$. Let $V$ be a neighborhood of $q$ such that for any $q'\in V$:

(1) $\nabla l_i(q')\in C\left(\tau,\frac12\pi-\frac12\theta_2\right)$,

(2) $\max\{l_i(q')\}<\min_{\gamma\not\in S(q)}\{l_\gamma(q')\}$.

Therefore, we have
\begin{align*}
    &\sum_\gamma e^{-\frac1Tl_\gamma(q')}\|\nabla\syst(q')\|=\left\|\sum_\gamma e^{-\frac1Tl_\gamma(q')}\nabla l_\gamma(q')\right\|\\
    &\ge \left\|\sum_i e^{-\frac1Tl_i(q')}\nabla l_i(q')\right\|-\left\|\sum_{\gamma\not\in S(q)} e^{-\frac1Tl_\gamma(q')}\nabla l_\gamma(q')\right\|\\
    &\ge \left\|\left\langle\sum_i e^{-\frac1Tl_i(q')}\nabla l_i(q'),\tau\right\rangle\right\|-\left\|\sum_{\gamma\not\in S(q)} e^{-\frac1Tl_\gamma(q')}\nabla l_\gamma(q')\right\|\\
    &>\sin(\frac12\theta_2)\sum_i\left\|\nabla l_i(q')\right\|e^{-\frac1T\max\{l_i(q')\}}-C(q)e^{-\frac1T\min_{\gamma\not\in S(q)}\{l_\gamma(q')\}}>0,
\end{align*}
for $T<T_0$ for appropriate $T_0>0$.
\end{proof}
We put a few lemmas after the following theorem, that are used in the proof.
\begin{theorem}[Semi-eutactic points]
For any semi-eutactic point $q\in\mgn$, there exist a neighborhood $V=V(q)$ of $q$ and $T_0=T_0(q)>0$ such that any $q'\in V$ is regular for $\syst$ for $T<T_0$.
\end{theorem}
\begin{proof}

By Lemma \ref{sesplit} below, we split the index set $I$ of $S(q)$ into a maximal eutactic index subset $I_{e}$ and the biased index complement $I_{b}$. Let $V(q)$ be the geodesic ball $B(q,t_0)$ for $t_0>0$ to be determined. Recall that $T^{\sys}_q\mathcal T=\spn_{I}\{\nabla l_i\}$ (and let us use $T^{\sys}$ for short). We let $$T^{e}=\spn_{I_{e}}\{\nabla l_i\} \text{ and } T^{e\perp}=\left(T^e\right)^\perp.$$ Also recall that $\theta_0$ is chosen such that $\angle\left(\nabla^2l_i(\tau,\cdot),\tau\right)<\frac{\pi}{2}-2\theta_0$ for all $i$.

By continuity, we choose $\theta_3,\theta_4$ and $d$ small enough such that

(1) $\theta_3<\theta_0$ and $C_2\max\{\theta_3,C_1\theta_4\}<\theta_0$, where $C_1,C_2$ are fixed constants to be introduced in case 2(b),

(2) $\max_{i\in I_e,\angle(\tau,T^{e\perp})\le\theta_3}|\left\langle\nabla l_i,\tau\right\rangle|<d$,

(3) If $\left\langle\nabla l_k,\tau\right\rangle\ge -2d$, then $\angle(\nabla l_k,\tau)\le\frac{\pi}{2}+\theta_4$ for $k\in I_{b}$,

and to satisfy other conditions introduced in the proof.

\noindent
\textbf{Case 1}: We have $\angle(\tau,T^{e\perp})\ge\theta_3$ (that is, $\tau$ not perpendicular to $T^e$ with angle bound).

By the semi-eutactic condition, since $\tau\not\in T^{e\perp}$, we have $\overline{\mathbb H(\tau)}\cap\{\nabla l_i(q)\}\neq\emptyset$ and $\mathbb H(-\tau)\cap\{\nabla l_i(q)\}\neq\emptyset$. In fact, in the compact region of the $\tau$'s in this case, there exists $D>0$ such that

$$\min_i\left\langle\nabla l_i(q),\tau\right\rangle<-D.$$

Let $$J:=\{j:\left\langle\nabla l_j(q),\tau\right\rangle=\min_i\left\langle\nabla l_i(q),\tau\right\rangle\},$$ and $$K=\overline{\mathbb H(\tau)}.$$

For any $q'\in V(q)$, $j\in J$ and $k\in K$, we have
\begin{align*}
    \frac{\left\|e^{-\frac1Tl_k(q')}\nabla l_k(q')\right\|}{\left\|e^{-\frac1Tl_j(q')}\nabla l_j(q')\right\|}&=e^{-\frac1T(l_k(q')-l_j(q'))}\frac{\left\|\nabla l_k(q')\right\|}{\left\|\nabla l_j(q')\right\|}\\
    &=e^{\frac1T(t\left\langle\nabla l_j(q),\tau\right\rangle-t\left\langle\nabla l_k(q),\tau\right\rangle+O(t^2))}\frac{\left\|\nabla l_k(q')\right\|}{\left\|\nabla l_j(q')\right\|}\\
    &<e^{-\frac12\frac1Tt_0D}\frac{\left\|\nabla l_k(q')\right\|}{\left\|\nabla l_j(q')\right\|}<\epsilon,
\end{align*}
for $t_0$ small enough.

Pick a $j\in J$, and consider the sum over $K^c$,
\begin{align*}
    \left\|\sum_{K^c}e^{-\frac1Tl_j(q')}\nabla l_j(q')\right\|&\ge-\left\langle\sum_{K^c}e^{-\frac1Tl_j(q')}\nabla l_j(q'),\tau\right\rangle\\
    &\ge-\left\langle\sum_Je^{-\frac1Tl_j(q')}\nabla l_j(q'),\tau\right\rangle\\
    &\ge De^{-\frac1Tl_j(q')}
\end{align*}

Put all terms together,
\begin{align*}
    \left\|\sum_{S(q)}e^{-\frac1Tl_i(q')}\nabla l_i(q')\right\|&\ge \left\|\sum_{K^c}e^{-\frac1Tl_i(q')}\nabla l_i(q')\right\|\\
    &-\left\|\sum_Ke^{-\frac1Tl_i(q')}\nabla l_i(q')\right\|\\
    &\ge D\sum_Je^{-\frac1Tl_j(q')}-r\epsilon e^{-\frac1Tl_j(q')}\left\|\nabla l_j(q')\right\|\\
    &>(D-rM\epsilon)e^{-\frac1Tl_j(q')},
\end{align*}
where $M=\max_{i;q'}\|\nabla l_i(q')\|$. Note that setting $t_0$ small, the sum above is $O(e^{-\frac1T(\sys(q)+\epsilon)})$, and the `tail', the sum over $S(q)^c$, is bounded by $e^{-\frac1T(\secsys(q)-\epsilon)}=o(e^{-\frac1T(\sys(q)+\epsilon)})$. Therefore, $\nabla\syst(q')\neq0$.

\textbf{Case 2}: We have $\angle(\tau,T^{e\perp})\le\theta_3$ (that is, when $\tau$ almost perpendicular to $T^e$).

Let $$\mathbf C=\cap_{k\in I_{b}}\mathbb H_{T^{\sys}}(\nabla l_k)=\{\tau:\left\langle\nabla l_k,\tau\right\rangle>0\text{\ for\ all\ }k\}$$ be the polygonal cone in $T^{\sys}$ consisting of vectors with positive inner products with $\nabla l_k$ for all $k\in I_{b}$, which is nonempty and convex by biased condition. Note that $\mathbf C\times T^{\sys\perp}$ is the set of such vectors in the total tangent space. We fix a $v\in \mathbf C$ and make $t$ sufficiently small so that $\left\langle\nabla l_k(q'),v\right\rangle$ is positively bounded from below by continuity.

\textbf{Sub-case 2(a)}: In addition, $\angle(\nabla l_k,\tau)\le\frac{\pi}{2}+\theta_4$ for all $k\in I_{b}$.

Let $J:=\{j\in I_{b}:\frac{\pi}{2}\le\angle(\nabla l_j,\tau)\le\frac{\pi}{2}+\theta_4\}\subset I_{b}$ and $K:=I_{b}\setminus J$ be the complement.

Let $T^{e\perp\sys}$ be the orthogonal complement of $T^e$ in $T^{\sys}$, then $$\left(\partial \mathbf C\times T^{\sys\perp}\right)\cap T^{e\perp}=\left(\partial \mathbf C\cap T^{e\perp\sys}\right)\times T^{\sys\perp},$$ which is trivial if and only if $\partial \mathbf C\cap T^{e\perp\sys}=0$ and $T^{\sys\perp}=0$ (and thus $T^{e\perp\sys}=T^{e\perp}$). In this case $J=\emptyset$ if $\theta_3$ and $\theta_4$ are small enough. Otherwise, we would have $$\angle\left(\partial \mathbf C,T^{e\perp\sys}\right)\le\angle\left(\tau,\partial \mathbf C\right)+\angle\left(\tau,T^{e\perp}\right)\le C_1\theta_4+\theta_3$$ by Lemma \ref{convex proj}. Contradiction!

If $J=\emptyset$, then $\tau\in \mathbf C\times T^{\sys\perp}$. Note that in this situation $\mathbf C\cap T^{e\perp}$ is nonempty. Let $\tau'$ be the projection of $\tau$ onto $T^{e\perp}$, then $$\tau'\in \left(\mathbf C\times T^{\sys\perp}\right)\cap T^{e\perp}$$ by the convexity of $\mathbf C\times T^{\sys\perp}$ and $\angle(\tau,\tau')<\theta_3<\theta_0$. 

If $J\neq\emptyset$, let $\tau'$ be the projection of $\tau$ onto $(\mathbf C\times T^{\sys\perp})\cap T^{e\perp}$, which is nontrivial. By Lemma \ref{convex proj}, $$\angle\left(\tau,\partial \mathbf C\times T^{\sys\perp}\right)\le\angle\left(\tau,\partial \mathbf C\right)<C_1\theta_4$$ and since $\angle(\tau,T^{e\perp})\le\theta_3$, by Lemma \ref{intersection angle}, $$\angle\left(\tau,\tau'\right)<C_2\max\{\theta_3,C_1\theta_4\}<\theta_0.$$

Therefore, for $i\in I_e\cup J$, since
$\nabla l_i(q)\perp\tau'$, $$\left\langle\nabla l_i(q'),\tau'\right\rangle=\left\langle\nabla l_i(q')-\nabla l_i(q),\tau'\right\rangle=\left\langle t\nabla^2l_i(\tau,\cdot)+O_i(t^2),\tau'\right\rangle.$$

Consider the angle, when $t$ is sufficiently small,
\begin{align*}
    \angle\left(t\nabla^2l_i(\tau,\cdot)+O_i(t^2),\tau'\right) & \le\angle\left(t\nabla^2l_i(\tau,\cdot)+O_i(t^2),\tau\right)+\angle\left(\tau,\tau'\right)\\
    & \le\frac{\pi}{2}-\frac32\theta_0+\theta_0=\frac{\pi}{2}-\frac12\theta_0,
\end{align*}
Thus, $\left\langle\nabla l_i(q'),\tau'\right\rangle>0$.

For $k\in K$, note that $\left\langle\nabla l_k(q),\tau'\right\rangle>0$ since $\tau'\in \intr(\mathbf C\times T^{\sys\perp})$ if $J=\emptyset$ and $\tau'\in\partial \mathbf C\times T^{\sys\perp}$ if $J\neq\emptyset$, then
\begin{align*}
    \left\langle\nabla l_k(q'),\tau'\right\rangle &= \left\langle\nabla l_k(q),\tau'\right\rangle + \left\langle t\nabla^2 l_k(q)(\tau,\cdot)+O_k(t^2),\tau'\right\rangle>0
\end{align*}

where the second term is positive for the same reason as above.

Each term has a positive projection onto $\tau'$ and thus the sum is nonzero.

\textbf{Sub-case 2(b)}: In addition, $\min_{k\in I_{b}}\left\langle\nabla l_k,\tau\right\rangle\le -2d$.

Let $J\subset I_{b}$ be the index set of the vectors that realize the minimum above.

For $i\in I_e$, note that $\left\langle\nabla l_j,\tau\right\rangle-\left\langle\nabla l_i,\tau\right\rangle\le -2d+d=-d$, then
\begin{align*}
    &\frac{\left\|\proj_v(e^{-\frac1Tl_i(q')}\nabla l_i(q'))\right\|}{\left\|\proj_v(e^{-\frac1Tl_j(q')}\nabla l_j(q'))\right\|}\\
    &=\frac{e^{-\frac1T(t\left\langle\nabla l_i,\tau\right\rangle+O_i(t^2))}} {e^{-\frac1T(t\left\langle\nabla l_j,\tau\right\rangle+O_j(t^2))} }\frac{\left\|\proj_v(t\nabla^2l_i(q)(\tau,\cdot)+O_i(t^2))\right\|}{\norm*{\proj_v(\nabla l_j(q'))}}\\
    &=te^{\frac1T(t\left\langle\nabla l_j,\tau\right\rangle-t\left\langle\nabla l_i,\tau\right\rangle+O_{ij}(t^2))}\frac{\norm*{\proj_v(\nabla^2l_i(q)(\tau,\cdot)+O_i(t))}}{\left\|\proj_v(\nabla l_j(q'))\right\|}\\
    &\le te^{-\frac12\frac1Ttd}\frac{\norm*{\proj_v(\nabla^2l_i(q)(\tau,\cdot)+O_{ij}(t))}}{\norm*{\proj_v(\nabla l_j(q'))}}
    <\epsilon,
\end{align*}
for $t$ sufficiently small.

By the choice of $v$, $\left\langle\nabla l_k(q'),v\right\rangle$ is bounded from below by a positive constant for all $k\in I_{b}$. Therefore,
\begin{align*}
    \left\|\sum_Ie^{-\frac1Tl_i(q')}\nabla l_i(q')\right\|&\ge\|\proj_v\sum_{I_e\cup I_{b}}e^{-\frac1Tl_i(q')}\nabla l_i(q')\|\\
    &\ge\left\|\proj_v\sum_{I_{b}}e^{-\frac1Tl_j(q')}\nabla l_j(q')\right\|\\
    &-\left\|\proj_v\sum_{I_{e}}e^{-\frac1Tl_i(q')}\nabla l_i(q')\right\|\\
    &\ge\left\|\proj_v\sum_{J}e^{-\frac1Tl_j(q')}\nabla l_j(q')\right\|\\
    &-\left\|\proj_v\sum_{I_{e}}e^{-\frac1Tl_i(q')}\nabla l_i(q')\right\|\\
    &>(1-r\epsilon)\left\|e^{-\frac1Tl_j(q')}\proj_v(\nabla l_j(q'))\right\|>0
\end{align*}

Each case corresponds to a compact region of the unit sphere, and the process in each case can be done continuously, hence the theorem follows.

\end{proof}

\begin{lemma}
\label{sesplit}
Let $\{v_i\}_{i\in I}$ be a semi-eutactic set, then the index set $I$ can be uniquely split into a maximal eutactic index subset $I_e$ and the biased index complement $I_{b}$. See Figure \ref{fig:split by eutacticity}.
\end{lemma}

\begin{proof}
Consider the convex hull of $\{v_i\}$, then the origin is on the boundary by definition. Let $P$ be the maximal face of the convex hull where the origin lives. Let $I_{e}$ be the index set of the maximal subset whose convex hull is $P$, then any hyperplane perpendicular to $\spn\{v_i\}$ passing through $\{v_i\}_{I_e}$ separates all other vectors on one side.
\end{proof}

\begin{figure}[ht]
    \centering
    \includegraphics[width=7cm]{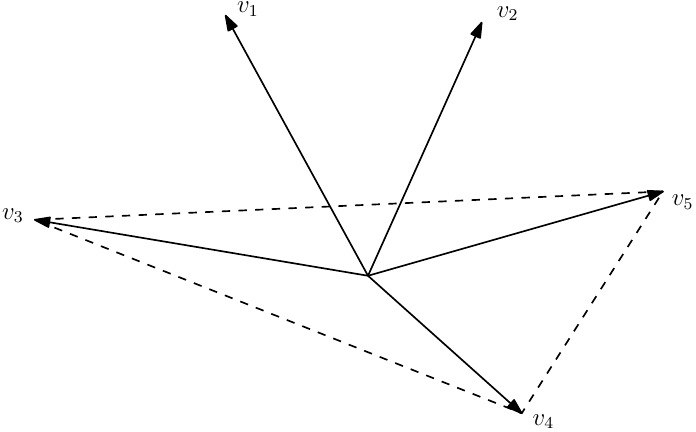}
    \caption{Split by eutacticity: in this example where $I=\{1,2,3,4,5\}$, we have $I_e=\{3,4,5\}$ and $I_b=\{1,2\}$}
    \label{fig:split by eutacticity}
\end{figure}

\begin{lemma}
Given a finite biased set of vectors $\{v_k\}\subset\mathbb R^n$, then for any $\theta>0$ small, there exists $d>0$ such that for any unit vector $\tau$ at least one of the following is true.

(1) $\angle(v_k,\tau)\le\frac{\pi}{2}+\theta$ for all $k$.

(2) $\min_k\left\langle v_k,\tau\right\rangle\le-2d.$
\end{lemma}
\begin{proof}
    This follows from continuity.
\end{proof}
\begin{lemma}
\label{intersection angle}
Let $V_1,V_2\subset\mathbb R^n$ be two linear subspaces with nontrivial intersection $V_0:=V_1\cap V_2$. Then there exists a constant $C$ depending on $\angle(V_1,V_2)$ such that for $\theta$ small enough, for any $v\in\mathbb R^n$ with $\angle(v,V_i)<\theta$ for $i=1,2$, we have $\angle(v,V_0)<C\theta$, i.e., $$C(V_1,\theta)\cap C(V_2,\theta)\subset C(V_0,C\theta).$$
\end{lemma}

\begin{proof}
It is obvious if one is a subspace of the other. Assume $\angle(V_1,V_2)>0$. Let $v_i$ be the projection of $v$ onto $V_i$, $i=0,1,2$ and $v'_i=v_i-v_0$, $i=1,2$. By the setup $\angle(V_1,V_2)=\min\{\angle(v'_1,v'_2),\pi-\angle(v'_1,v'_2)\}>0$ and $\angle(v,V_i)=\angle(v,v_i)$ for $i=0,1,2$, it suffices to prove the lemma in $\mathbb R^3$. Let $C_i$ be the great circle on $S^2$ spanned by $v_i$ and $v_0$, $i=1,2$, then $C_1\cap C_2=\pm v_0$, and $\min\{\angle(v'_1,v'_2),\pi-\angle(v'_1,v'_2)\}=\angle(C_1,C_2)$. Therefore, there exists $C>0$ depending on $\angle(C_1,C_2)$ such that $\angle(v,v_0)=d_{S^2}(v,v_0)< C\max_{i=1,2}{d_{S^2}(v,C_i)}=C\max_{i=1,2}{\angle(v,v_i)}<C\theta$ if $v$ is close enough to $C_1$ and $C_2$.
\end{proof}

\begin{figure}[ht]
    \centering
    \includegraphics[width=5cm]{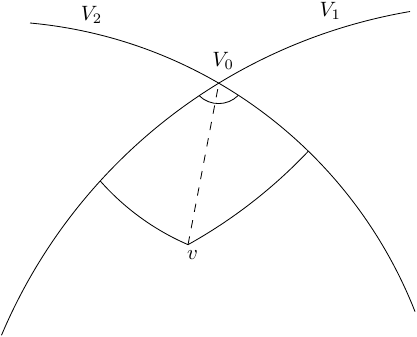}
    \caption{Projectivized view of the settings in Lemma \ref{intersection angle}}
\end{figure}

\begin{lemma}
\label{convex proj}
Let $\{v_i\}_{i\in I}\subset\mathbb R^n$ be a finite biased set of vectors, and $\mathbf C=\cap_i\mathbb H(v_i)$, then there exists $\theta$ small enough, such that for any vector $\tau\not\in \mathbf C$ with $\angle(v_i,\tau)\le\frac{\pi}{2}+\theta$ for all $i$, there exists a unique unit vector $u\in\partial\mathbf C$ such that $\angle(u,\tau)=\min_{v\in\partial\mathbf C}\angle(v,\tau)$.

Furthermore, $\angle(u,\tau)<C\theta$ for constant $C$.
\end{lemma}

\begin{proof}

It is clear that $v\in \mathbf C$ if and only if $\left\langle v_i,v\right\rangle\ge0$ for all $i$, then for $v_1,v_2\in \mathbf C$, for $0\le\lambda\le1$, $\left\langle v_i,\lambda v_1+(1-\lambda)v_2\right\rangle\ge0$, showing that $\mathbf C$ is radial and convex. Note that $\mathbf C$ is contained in some half space. Now consider the angle between a vector and $\tau$ as a function on $\mathbf C\cap S^{n-1}$. If $\tau\not\in \mathbf C\cap S^{n-1}$, $\angle(\cdot,\tau)=d_{S^{n-1}}(\cdot,\tau)$ is minimized over $\mathbf C\cap S^{n-1}$ by a unique vector $v\in\partial \mathbf C\cap S^{n-1}$ by convexity if $\tau$ is close enough to $\partial \mathbf C$.

The second part is due to Lemma \ref{intersection angle} and the finiteness of faces.
\end{proof}
\begin{figure}[ht]
    \centering
    \includegraphics[width=5cm]{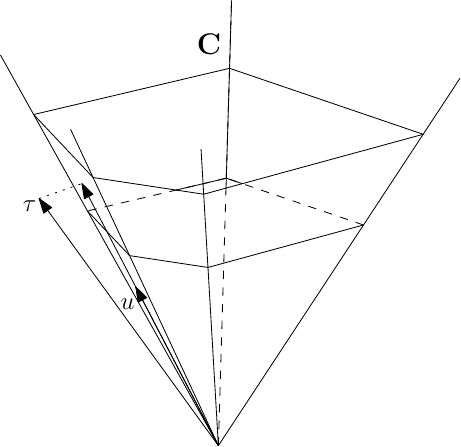}
    \caption{Projection onto the polygonal cone of a vector near the convex hull}
\end{figure}
\begin{definition}
\label{convex proj def}
With the same settings in the above lemma, define $$\angle(\tau,\partial\mathbf C)=\angle(u,\tau) \text{ and\ } \proj_{\mathbf C}\tau=\proj_{u}\tau.$$
\end{definition}

\begin{remark}
We can reduce the condition on $\mathbf C$ to any radial and convex subset for Lemma \ref{convex proj} and Definition \ref{convex proj def}.
\end{remark}
\begin{lemma}
Let $\mathbf C=\cap\mathbb H_i\neq\emptyset$ be the intersection of finitely many half spaces. Suppose $V$ is a linear subspace that has nontrivial intersection with $\mathbf C$, then for any $v\in\mathbb R^n$ and for $\theta$ small, if $\angle(v,\partial \mathbf C)<\theta$ and $\angle(v,V)<\theta$, then $$\angle(v,\mathbf C\cap V)\le\angle(v,\partial \mathbf C\cap V)<C\theta.$$
\end{lemma}
\begin{proof}
Let $\tau'$ be the projection of $\tau$ onto $\partial \mathbf C$ and $f$ be a face of $\partial \mathbf C$ such that $\tau'\in f$, then we claim $f\cap V\neq\emptyset$ if $\theta$ is chosen small enough. Otherwise, $\angle(f,V)\le\angle(v,f)+\angle(v,V)<2\theta$, which is a contradiction if $2\theta<\min_{f\in F}\angle(f,V)$, where $F=\{f, f \text{ is a face of }\partial \mathbf C \text{ and }f\cap V=\emptyset\}$. Note that $f\cap V$ is radial and convex, $\angle(\tau,\partial \mathbf C\cap V)\le\angle(\tau,f\cap V)\le\angle(\tau,\proj_{f\cap V}\tau')<C\theta$ by Lemma \ref{intersection angle}.
\end{proof}

\subsection{Near the boundary}
\label{boundary}
\begin{definition}
The $\epsilon$-thin part and $\epsilon$-thick part of the moduli space are defined as
    $$\mathcal M^{\le\epsilon}=\{X\in\mathcal M:\sys(X)\le\epsilon\},$$
    $$\mathcal M^{\ge\epsilon}=\{X\in\mathcal M:\sys(X)\ge\epsilon\}.$$
\end{definition}
\begin{theorem}[Points near the boundary]
There exist $\epsilon>0$ and $T_0>0$ such that any $q\in\mathcal M^{\le\epsilon}$ is a regular point for $\syst$ for all $T<T_0$. In fact, for any $\beta$ of length $\le\epsilon$, we have $\left\langle\nabla l_\beta,\nabla\syst\right\rangle>0$.
\end{theorem}
\begin{proof}
Let $q\in\mathcal M^{\le\epsilon}$ and $S(q)=\{\gamma_1,\cdots,\gamma_r\}$. By the collar lemma or Corollary 4.1.2 in \cite{buser2010geometry}, $\gamma_i$'s are disjoint. Let $\gamma_I=\cup\gamma_i$. We have the Weil-Petersson pairing of $\nabla l_i$ with $\nabla l_\gamma$ for $\gamma$ classified into the following three types.

Type 1: $\gamma=\gamma_j$.
$$\left\langle\nabla l_\gamma,\nabla l_i\right\rangle>\left\|\nabla l_i\right\|\delta_{ij}.$$

Type 2: $\gamma\cap\gamma_I=\emptyset$.
$$\left\langle\nabla l_\gamma,\nabla l_i\right\rangle>0.$$

Type 3: $i(\gamma,\gamma_I)>0$.

By the collar lemma, $l_\gamma>x(\epsilon)$, where $x(\epsilon)=2\arcsinh(\frac{1}{\sinh\frac{\epsilon}{2}})>-2\log\epsilon$.

Therefore, choose any $i$,
\begin{align*}
    \norm*{\sum_\gamma e^{-\frac1Tl_\gamma}\nabla l_\gamma}&\ge\norm*{\sum_{i(\gamma,\gamma_I)=0} e^{-\frac1Tl_\gamma}\nabla l_\gamma}-\norm*{\sum_{i(\gamma,\gamma_I)>0} e^{-\frac1Tl_\gamma}\nabla l_\gamma}\\
    &\ge \norm*{\proj_{\nabla l_i}\left( \sum_{i(\gamma,\gamma_I)=0} e^{-\frac1Tl_\gamma}\nabla l_\gamma \right)} - \norm*{\sum_{i(\gamma,\gamma_I)>0} e^{-\frac1Tl_\gamma}\nabla l_\gamma}\\    
    &\ge e^{-\frac1Tl_i}\norm*{\nabla l_i}-Ce^{-\frac1T x(\epsilon)}\\
    &\ge \frac12\sqrt{\epsilon}e^{-\frac1T\epsilon}-Ce^{\frac2T\log\epsilon}.
\end{align*}

In order for the lower bound above to be positive, it suffices to make $T$ satisfy
\begin{align*}
    T<\frac{-2\log\epsilon-\epsilon}{\log(2C)-\frac12\log\epsilon}
\end{align*}
Note that
\begin{align*}
    \frac{-2\log\epsilon-\epsilon}{\log(2C)-\frac12\log\epsilon}\to 4,
\end{align*}
as $\epsilon\to0^+$. Therefore, when $\epsilon<\epsilon_0$ is small enough, a uniform upper bound $T_0$ for $T$ can be chosen such that for any $q\in\mathcal M^{\le\epsilon_0}$, $\syst(q)\neq0$ for all $T<T_0$.
\end{proof}

\begin{remark}
By the critical points attracting property, all critical points of $\syst$ live in the thick part of the moduli space.
\end{remark}
\begin{remark}
\label{boundarydirection}
    The positive pairing in the theorem implies that when $T<T_0$, $\nabla\syst$ points transversely inward to the $\sys$-level sets $\mathcal M^{=\epsilon'}$ in $\mathcal M^{\le\epsilon}$. Therefore, we have the following.
\end{remark}
\begin{corollary}
For $\epsilon$ small enough, $\mathcal M^{\ge\epsilon}$ is a deformation retract of $\mathcal M$.
\end{corollary}

\section{Extension onto the boundary}
\label{extension}
Recall that the Deligne-Mumford boundary $\partial_{\textit{DM}}\mathcal M$ of the moduli space $\mathcal M(X)$ is the space of all nodal hyperbolic surfaces modeled on $X$. We may use $\partial\mathcal M$ for simplicity. The Deligne-Mumford compactification is the union $$\overline{\mathcal M}=\mathcal M\cup\partial \mathcal M.$$ It is precisely the quotient of the augmented $\teich$ space $\overline{\mathcal T}$, which is obtained by allowing length parameters equal to 0, by the mapping class group. A stratum in the compactification is a maximal connected subset in which the points are all homeomorphic. See \cite{deligne1969irreducibility} and \cite{hubbard2014analytic} for more details.

Let $X_0\in\partial\mgn$ with the pinched curve set $S$ and call a neighborhood \textit{stratifically\ closed} if the pinched curve set of any point is a subset of $S$. Expand $S$ to a pants decomposition $\overline{S}=\{\alpha_i\}$, then the Fenchel-Nielsen coordinates are given by the associated length and twist parameters $\{l_i,\tau_i\}$. For a small neighborhood $U$ of $X$ that is stratifically closed, let $(U,\{k_ie^{i\tau_i}\})$ be a chart, with $k_ie^{i\tau_i}=l_i^\chi e^{i\tau_i}$ being the transition maps on the overlap, where $\chi=\frac12$ if $\alpha_i\in S$ and 1 otherwise. We call such a chart a \textit{nodal chart} at $X_0$. Different extensions are $C^\infty$-compatible by analytical compatibility of the Fenchel-Nielsen coordinates associated to different pants decompositions.

\begin{remark}
    The differential structure at such a nodal surface, defined using the root geodesic-length function $l_i^{\frac{1}{2}}$ (but not $l_i$) for each $\alpha_i$, is quite natural. Note that near the boundary, one has $\norm*{\nabla l_i^{\frac{1}{2}}}\approx\frac{1}{2\pi}$. Due to this estimate, on many occasions, it would be easier to consider $\nabla l_i^{\frac{1}{2}}$ instead of $\nabla l_i$, for example, in \cite{wolpert2008behavior}.
\end{remark}

Let $S=\{\beta_1,\cdots,\beta_s\}$ be a set of mutually disjoint simple closed geodesics (not necessarily shortest) on $X$. When $l_{\beta_i}$ tends to 0 for all $i$, by definition we have $X\to X_0\in\partial\mgn$ ending somewhere in the respective stratum. The limit nodal surface $X_0$ has a different topology and may not be connected off the nodes.

The tangent space of the compactification at $X_0\in\partial\mgn$ can be decomposed as $$T(X_0)=T^{\textit{Str}}(X_0)\oplus T^{\textit{Nod}}(X_0),$$ where stratum tangent subspace $T^{\textit{Str}}(X_0)$ is given by the length and twist parameters associated to $\overline{S}\setminus S$ and the nodal tangent subspace $T^{\textit{Nod}}(X_0)$ is given by those for $S$.

Under a nodal chart at some $X\in\partial\mgn$ with pinched curve set $S$ as above, the geodesic-length functions $l_{\beta_i}=k_{\beta_i}^2$ are $C^\infty$.

To extend $\syst$ to $\partial\mgn$, note that for fixed $T>0$, we have $$e^{-\frac1Tl_{\beta_i}}\to 1 \text{ as } l_{\beta_i}\to 0.$$

For any $\gamma$ that crosses some $\beta_i$, since $\sinh\frac{l_{\beta_i}}{2}\sinh\frac{l_\gamma}{2}>2\arcsinh1$, we have $$l_\gamma\to\infty \text{ and } e^{-\frac1Tl_\gamma}\to 0 \text{ as } l_{\beta_i}\to 0.$$
\begin{definition}
    For $X\in\partial\mgn$, let 
    \begin{align*}
        \syst(X)=&-T\log\left(s+\sum_{\gamma \text{ s.c.g. on } X} e^{-\frac1Tl_\gamma(X)}\right)\\
        =&-T\log\left(s+\sum_i e^{-\frac1T\syst(X_i)}\right),
    \end{align*}
    where $s=\# S$.
\end{definition}
It is known from previous discussion that $\syst$ is at least $C^2$-continuous on each stratum. Note that an extension of the Weil-Petersson metric from $\mgn$ to the compactification $\mgnb$, that equals the Weil-Petersson metric when restricted onto each stratum, does not exist. But given the continuity of $\syst|_{\mgnb}$, following the stratum-wise gradient flow, we know that it attains its global minimum on maximally pinched surfaces, and therefore is positively bounded on $\partial\mgn$. We write this as the following lemma.
\begin{lemma}
\label{positivityofsyst}
    The function $\syst$ on $\mgnb$ is bounded by positive numbers.
\end{lemma}

\begin{lemma}
\label{derivativeextension}
    For $T$ sufficiently small, $\syst\colon \mgnb\to\mathbb R$ is $C^2$-continuous, under the differential structure given by nodal charts.
\end{lemma}
\begin{proof}
    It is already shown that $\syst$ is $C^2$-continuous in the interior and continuous globally by its definition. Let $X\to X_0\in\partial\mgn$ with pinched curve set $S=\{\beta_1,\cdots,\beta_s\}$. Recall that the first derivative in $\mgn$ is given by $$\syst'=\frac{\sum_\gamma e^{-\frac1Tl_\gamma}l_\gamma'}{\sum_\gamma e^{-\frac1Tl_\gamma}},$$ and the second derivative
\begin{align*}
\syst''= \frac1T\left(\frac{\sum_\gamma e^{-{\frac1T}l_\gamma}l_\gamma'}{\sum_\gamma e^{-{\frac1T}l_\gamma}}\right)^2-\frac1T\frac{\sum_\gamma e^{-{\frac1T}l_\gamma}(l_\gamma')^2}{\sum_\gamma e^{-{\frac1T}l_\gamma}} +\frac{\sum_\gamma e^{-{\frac1T}l_\gamma}l_\gamma''}{\sum_\gamma e^{-{\frac1T}l_\gamma}}.
\end{align*}

Thus, it suffices to show that $e^{-\frac1T}l'$ and $e^{-\frac1T}l''$ converge when the base point approaches the boundary. There are three types of geodesics (cf. Figure \ref{fig:all types of gamma}), and we consider each of them separately.
\begin{figure}[ht]
    \centering
    \includegraphics[width=10cm]{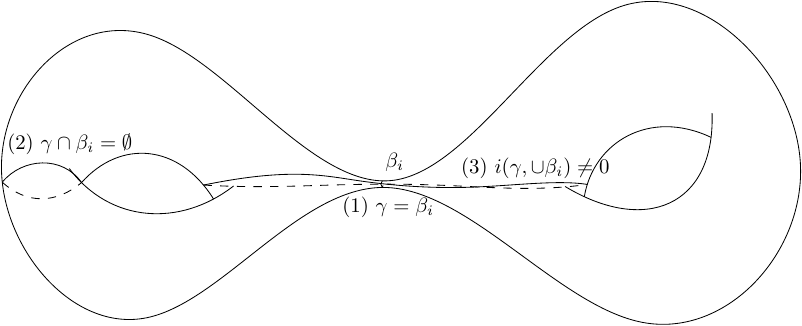}
    \caption{Three types of $\gamma$}
    \label{fig:all types of gamma}
\end{figure}

Type 1: $\gamma=\beta_i$. It is clear by definition that $l=k_{\beta_i}^2$ is $C^\infty$.

Type 2: $\gamma\cap\beta_i=\emptyset$. We have $l_\gamma=2\arccosh(\frac{\text{tr}(\Pi A_i)}{2})$ by a tracking algorithm by unfolding the surface, where each $A_i$ is a rotation or translation matrix smooth in Fenchel-Nielsen coordinates and not involving $\tau_{\beta_i}$'s.

Type 3: $i(\gamma,\cup\beta_i)\neq0$. Since they intersect as closed geodesics, we have $$l_\gamma\ge-2\log l_{\beta_i}.$$ Let $\gamma_i$'s be a few geodesics disjoint from $\cup \beta_i$, such that $\{\lambda_i=\nabla l_{\beta_i}^{\frac12},J\lambda_i,\nabla l_i\}$ is a local frame.

For such $\gamma$, for the first directional derivatives, note that by Theorem \ref{lengthgradientnorm}
\begin{align*}
    e^{-\frac1Tl_\gamma}\left|\left\langle\nabla l_\gamma,\lambda_i\right\rangle\right|,e^{-\frac1Tl_\gamma}|\left\langle\nabla l_\gamma,J\lambda_i\right\rangle|&\le e^{-\frac1Tl}\|\nabla l_\gamma\|\cdot\|\lambda_i\|\\
    \le e^{-\frac1Tl_\gamma}\left(c\left(l_\gamma+l_\gamma^2e^{\frac{l_\gamma}{2}}\right)\right)^{\frac12}\cdot c\to0,
\end{align*}
and
\begin{align*}
    e^{-\frac1Tl_\gamma}\left|\left\langle\nabla l_\gamma,\nabla l_i\right\rangle\right|&\le e^{-\frac1Tl}\|\nabla l_\gamma\|\cdot\|\nabla l_i\|\\
    &\le e^{-\frac1Tl_\gamma}\left(c\left(l_\gamma+l_\gamma^2e^{\frac{l_\gamma}{2}}\right)\right)^{\frac12}\cdot \left(c\left(l_i+l_i^2e^{\frac{l_i}{2}}\right)\right)^{\frac12}\to0,
\end{align*}
as $l_{\beta_i}\to0$.

For the second directional derivatives, note that by Theorem \ref{connection}
\begin{align*}
    \norm*{e^{-\frac1Tl_\gamma}XYl_\gamma}&= \norm*{e^{-\frac1Tl_\gamma}\left(H(l_\gamma)(X,Y)+\left\langle D_XY,\nabla l_\gamma\right\rangle\right)}\\
    &\le e^{-\frac1Tl_\gamma}\left(\norm*{H(l_\gamma)(X,Y)}+\norm*{D_XY}\cdot\norm*{\nabla l_\gamma}\right)\\
    &\le e^{-\frac1Tl_\gamma}\left(\left(4c(1+l_\gamma e^{\frac{l_\gamma}{2}}\right)+O\left(l_\gamma^3\right)\right)\norm*{X}\cdot\norm*{Y}+\frac{c}{l_{\beta_i}^\frac12}\cdot\norm*{\nabla l_\gamma})\to0,
\end{align*}
where $X,Y$ are two local vector fields near $X_0$.

Therefore, $e^{\frac1Tl_\gamma}$ is $C^2$-continuous on the augmented $\teich$ space, and the sum of it over the orbit of $\gamma$ is $C^2$-continuous on the Deligne-Mumford compactification.

This completes the proof.

\end{proof}

The critical point set $\crit(\syst|_{\mgnb})$ is relatively clear, restricted in $\mgn$. For critical points in a boundary stratum $\mathcal S\subset\partial\mgn$, they correspond exactly to the critical points for $\syst|_{\mathcal M_{\mathcal S}}$, where $\mathcal M_{\mathcal S}$ is the moduli space that is canonically isomorphic to $\mathcal S$, as the $\syst$ functions differ only by a constant up to the same homeomorphism of $\mathbb R$, under the identification $\mathcal S\leftrightarrow\mathcal M_{\mathcal S}$. The $T^{\textit{Str}}$-directional derivatives are 0 since $\nabla\syst|_{\mathcal S}$ is parallel to $\nabla\syst|_{\mathcal M_{\mathcal S}}$ under the identification, and so are $T^{\textit{Nod}}$-directional derivatives following Remark \ref{boundarydirection}.

On a nodal chart $(U,\{k_ie^{i\tau_i}\})$, the second derivative matrix is

\begin{align*}
\begin{pmatrix}
2I & 0\\
0 & \ast
\end{pmatrix},
\end{align*}
where $\ast$ is the second derivative matrix of $\syst$
restricted on $T^{\textit{Str}}$ that is nondegenerate. This shows that critical points on $\partial\mgn$ are nondegenerate.

Based on the discussion above, we can give a description of critical points in the boundary $\partial\mgn$. Let $\mathcal S\subset\partial\mgn$ be a stratum in the boundary, that is canonically isomorphic to the product $\prod\mathcal M_i$ of moduli spaces. Any critical point $X\in\mathcal S$ is a nodal surface and has the corresponding decomposition $X=\cup X_i$ along the nodes, such that each $X_i$ is critical in $\mathcal M_i$, and $$\ind(X)=\sum\ind(X_i),$$ where $\ind(X_i)$ is the index of $X_i$ for $\syst|_{\mathcal M_i}$. In this way, one can decompose a critical point in the boundary to smaller surfaces that are critical in their respective moduli spaces. The converse of the statement also holds. One can also construct a critical point by connecting smaller critical points along nodes.

\begin{figure}[ht]
    \centering
    \includegraphics[width=9cm]{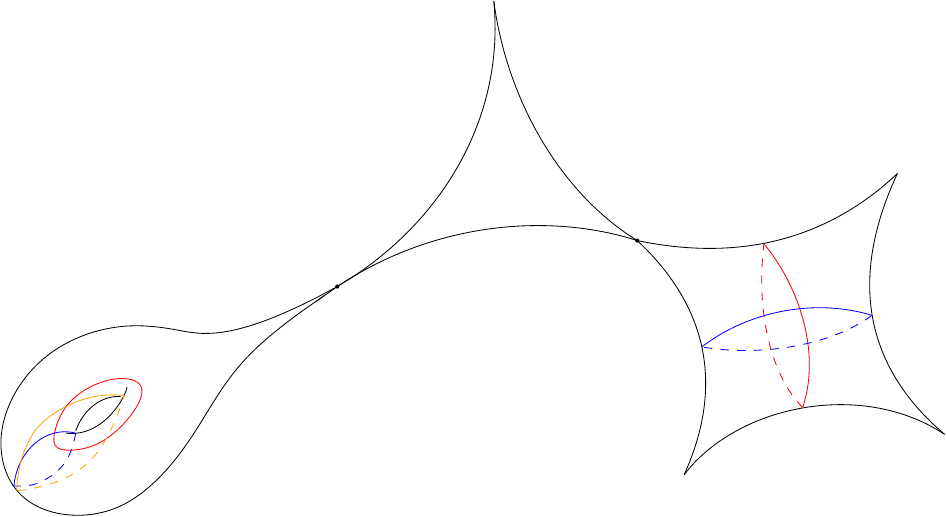}
    \caption{A critical point in $\partial\mathcal M_{1,4}$ that is the union of smaller critical points in $\mathcal M_{1,1},\mathcal M_{0,3}$ and $\mathcal M_{0,4}$ respectively along nodes}
    \label{fig:critical point in boundary}
\end{figure}

The surface $X$ in Figure \ref{fig:critical point in boundary} lives in $\mathcal S=\mathcal S_1\times\mathcal S_2\times\mathcal S_3\subset\partial\mathcal M_{1,4}$, where $\mathcal S_1\cong\mathcal M_{1,1}$, $\mathcal S_2\cong\mathcal M_{0,3}$ and $\mathcal S_3\cong\mathcal M_{0,4}$. Each subsurface $X_i\in\mathcal S_i$ off the nodes is a critical point for $\syst|_{\mathcal M_{\mathcal S_i}}$ with $S(X_i)$ marked in colors. The index of $X$ is $$\ind(X)=\sum\ind(X_i)=1+0+1=2.$$

\begin{theorem}
    $\syst\colon \mgnb\to\mathbb R$ is Morse with the given differential structure.
\end{theorem}
Part II of the main theorem follows from the definition of the extension by direct calculation. Part III holds for the extended $\syst$ for similar reasons as in Subsections \ref{behavior}-\ref{boundary}.

\section{Weil-Petersson gradient flow}
\label{case}
As the systole function is only continuous, one cannot define the Weil-Petersson gradient flow. Although the Weil-Petersson gradient `direction' at each point can be considered instead, it turns out that this is a discontinuous distribution in the tangent space. The $\syst$ functions are $C^2$-continuous, so the gradient vector is defined in $\tgn$ wherever the Weil-Petersson metric is defined.

In the last section, the proof of Lemma \ref{derivativeextension} also shows that the Weil-Petersson gradient vector of $\syst$ at any point in $\mgnb$ is well defined, although the connection is not. It is easy to see from the definition of the extension that starting with a direction at a point in any stratum, the gradient flow will locally stay in the same stratum and never go to an upper stratum.

Recall that the gradient vector of $\syst$ is given by $$\nabla\syst(X)=\frac{\sum_\gamma e^{-\frac1Tl_\gamma(X)}\nabla l_\gamma(X)}{\sum_\gamma e^{-\frac1Tl_\gamma(X)}}.$$

Suppose that a surface $X_0$ in the boundary of the augmented $\teich$ space $\overline{\mathcal T}_{g,n}$ is given with a pinched curve $\beta$, we can split the numerator above into three parts by contribution: (1) $\beta$, (2) curves disjoint from $\beta$, and (3) curves crossing $\beta$. Along the pinching ray of $\beta$ to the boundary of the augmented $\teich$ space, the contribution (3) is negligible and dominated by (1), as shown in the calculation in the proof of Lemma \ref{derivativeextension}. Along the downward flow line of $\syst$, using the nodal chart, the distance to the boundary locally satisfies
$$\frac{d}{dt}\left(l_\beta^{\frac{1}{2}}\right)=-\sqrt{\frac{2}{\pi}}l_\beta^{\frac12}+O\left(l_\beta^2\right),$$
when $l_\beta<\epsilon$. Therefore, $l_\beta\asymp e^{-2\sqrt{\frac{2}{\pi}}t}$.

Part IV follows from the construction and observation above.

An example of the Weil-Petersson gradient flow of $\syst$ on $\overline{\mathcal M}_{1,1}$ will be given in the next section.

\section{Applications}
\label{applications}

\subsection{Cell decomposition of $\mgnb$ and $\mgn$}
Although the Weil-Petersson metric blows up when it approaches the Deligne-Mumford boundary $\partial\mgn$, the gradient flow of $\syst$ is well defined. If one considers the $\syst$ functions on a finite manifold cover of $\mgnb$, by a finite quotient $G$ of the mapping class group, a $G$-invariant Morse handle decomposition can be constructed, which places a $k$-handle at each index $k$ critical point.

The $G$-invariance of the $\syst$-Morse handle decomposition defines a $G$-action on the set of handles. For a $k$-handle $H^k$ at a critical point $q_k$, one can take the quotient of $\text{Orb}(H^k)$ by $G$. This gives a quotient handle $H^k/\text{Stab}_G(q^k)$ at $\pi(q^k)$, where $\pi$ is the quotient map. One may compare this with an orbifold chart for the orbifold locus containing $\pi(q^k)$.

Note that a $k$-handle $H^k$ is diffeomorphic to a $k$-cell times a codimension $k$ disk. On the cover, for every $H^k$, a $G$-invariant refinement of $H^k$ by cells, that is compatible with orbifold loci and attaching maps, gives a cell decomposition of the cover that descends to a cell decomposition of $\mgnb$. Alternatively, one may directly take a refinement of the Morse `quotient handle' decomposition by cells to get a cell decomposition of $\mgnb$.

\subsection{Moduli of elliptic curves}

The moduli space $\mathcal M_{1,1}$ of elliptic curves, or equivalently, complete hyperbolic once-punctured tori, has been extensively studied by many. With Morse theory, we are able to obtain a new perspective.

There are two special points in $\mathcal M_{1,1}$. Let $X_1$ be the hyperbolic representative in the conformal class of the once-punctured torus obtained by gluing parallel sides of a punctured square, and $X_2$ be that for a punctured diamond with a $\frac{\pi}{3}$ interior angle. They are exactly the two $\syst$-critical points, as well as the orbifold points: $X_1$ is a double point, and has Morse index 1; $X_2$ is a triple point, and has Morse index 2.

The Deligne-Mumford boundary of $\mathcal M_{1,1}$ consists of a single point $X_0$, that is the once-noded once-punctured torus, that can be identified with a three-time-punctured sphere. $X_0$ is a $\syst$-critical point as well as an orbifold point. It is a double point, and its Morse index is 0.

Since $\overline{\mathcal M}_{1,1}$ is an orbifold with a triple point and two double points, we shall take the 6-sheeted cover, that we denote by $\overline{\mathcal M}_{1,1}^6$, to resolve the singularities. By Part IV of the main theorem, the Weil-Petersson gradient flow of $\syst$ is well defined on $\overline{\mathcal M}_{1,1}$, therefore, on $\overline{\mathcal M}_{1,1}^6$ as well, as shown below.

\begin{figure}[ht]
    \centering
    \includegraphics[width=6.5cm]{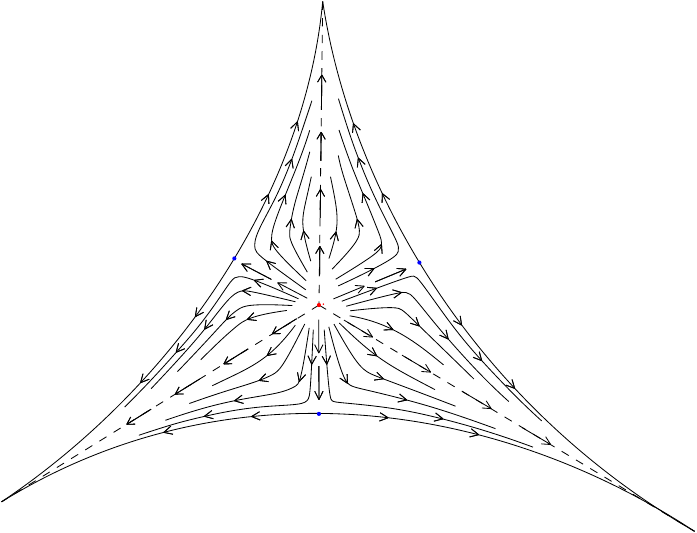}
    \caption{The `front view' of Weil-Petersson gradient flow of $\syst$ on $\overline{\mathcal M}_{1,1}^6$; the `back view' looks the same}
\end{figure}

Let $\pi\colon \overline{\mathcal M}_{1,1}^6\to \overline{\mathcal M}_{1,1}$ be the covering projection. Let $\pi^{-1}(X_2)=\{\alpha,\beta\}$, $\pi^{-1}(X_1)=\{a_1,a_2,a_3\}$ and $\pi^{-1}(X_0)=\{b_1,b_2,b_3\}$, then up to a choice of orientation, the Morse differentials are given by

\begin{align*}
    &\partial_2\alpha=a_1+a_2+a_3,\\
    &\partial_2\beta=-a_1-a_2-a_3,\\
    &\partial_1a_1=b_2-b_3,\\
    &\partial_1a_2=b_3-b_1,\\
    &\partial_1a_3=b_1-b_2.
\end{align*}

This defines the $\syst$-Morse chain complex
$$\begin{tikzcd}
0 \arrow{r} &\mathbb Z^2 \arrow{r}{\partial_2} &\mathbb Z^3 \arrow{r}{\partial_1} &\mathbb Z^3 \arrow{r} &0.
\end{tikzcd}$$

It would not be difficult to compute the Morse homology of $\overline{\mathcal M}_{1,1}^6$.
\begin{align*}
    H_2(\overline{\mathcal M}_{1,1}^6)=\mathbb Z,\\
    H_1(\overline{\mathcal M}_{1,1}^6)=0,\\
    H_0(\overline{\mathcal M}_{1,1}^6)=\mathbb Z.
\end{align*}

Using transfer injection on homology groups, it follows that $H_k(\overline{\mathcal M}_{1,1};\mathbb Q)=\mathbb Q$ when $k=0,2$ and 0 otherwise.

The figure below shows the $\syst$-Morse handle decomposition.

\begin{figure}[ht]
    \centering
    \includegraphics[width=6.5cm]{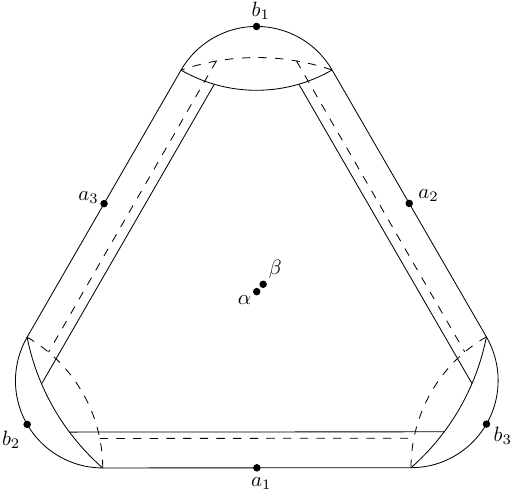}
    \caption{The $\syst$-Morse handle decomposition of $\overline{\mathcal M}_{1,1}^6$ with respect to the Weil-Petersson metric}
\end{figure}

\subsection{Calculating first homology groups of moduli spaces of curves via Morse theory}
The title of this subsection imitates the title of \cite{arbarello1998calculating} that calculates cohomology via algebraic geometry. Although more properties of the $\syst$ functions are needed for calculating higher degree (co)homology groups, the first (co)homology is within reach. We cite the theorem in the paper \cite{chen2025} that is being written that computes the first rational homology using the $\syst$ functions. This result can be found in earlier literature, for example, \cite{boggi2000galois}, \cite{farb2011primer}.
\begin{theorem}
    $H_1(\mgnb;\mathbb Q)=0$.
\end{theorem}

\subsection{Low degree homology groups}
    For convenience, we may say that a moduli space $\mgn$ or the compactification $\mgnb$ is \textit{large} (or \textit{small}), if its dimension is large (or small).
\begin{definition}
    A critical point $X\in\mgnb$ is called \textit{primitive}, if it is in $\mgn$. That is, $X$ is a smooth surface.
\end{definition}
In \cite{chen2023c2}, the following property of the $\syst$ functions is shown, as a result of hyperbolic geometry.
\begin{theorem}
    For any $k\in\mathbb N$, there are only finitely many primitive $\syst$-critical points of index at most $k$ in $\bigsqcup_{(g,n)}\mgnb$.
\end{theorem}
Another way to state this theorem is that there are no primitive critical points of index at most $k$ in any large $\mgnb$ (or when either $g$ or $n$ is large). In this case, all critical points of index at most $k$ live in the boundary $\partial\mgn$, and can be constructed with primitive critical points in smaller $\mgnb$ due to Part III of the main theorem, allowing one to apply induction.

Applying Morse theory to $\syst$ on $\mgnb$, we have the following.

\begin{theorem}
\label{generalrank}
For any $k$, except for finitely many pairs $(g,n)$, the inclusion $$i\colon \partial\mgn\to\mgnb$$ induces an isomorphism
        $$i_*\colon H_k(\partial \mgn;\mathbb Q) \xrightarrow{\cong} H_k(\mgnb;\mathbb Q).$$
\end{theorem}
\begin{proof}
    It is clear that $i_*$ is tautologically equivariant with respect to any finite manifold covering map of $\mgnb$. This follows from the fact that the Morse chain complex coincides for both $\mgnb$ and $\partial\mgn$.
\end{proof}


\newpage
\section*{Commonly Used Notations}
\begin{table*}[htbp]
\begin{center}
     \begin{tabular}{r c p{12cm} }
\toprule
&$J$ & multiindex; Weil-Petersson almost complex structure\\
&$K$ & multiindex\\
&$F_J$ & fan associated to $J$\\
&$D$ & constant; Weil-Petersson connection\\
&$\nabla$ & (Weil-Petersson) gradient vector\\
&$\norm*{\cdot}$ & (Weil-Petersson) norm\\
&$\gamma,\gamma_i$ & closed geodesic\\
&$\beta_i$ & pinched curve\\
&$l_\gamma,l_i$ & length function associated to $\gamma$, $\gamma_i$\\
&$\sys$ & systole function\\
&$\secsys$ & second systole function\\
&$\syst$ & defined to be $-T\log\sum_{\gamma \text{ s.c.g. on } X} e^{-\frac1Tl_\gamma(X)}$\\
&$S$ & pinched curve set\\
&$S(X)$ & set of shortest geodesics on a surface\\
&$c_X(L)$ & number of closed geodesics of length $\le L$ on $X$\\
&$r$ & number of shortest geodesics\\
&$I$ & index set for $S(X)$\\
&$I_{e}$ & subset of $I$ indexing the maximal eutactic subset\\
&$I_{b}$ & complement of $I_{e}$ in $I$ when $X$ is semi-eutactic\\
&$\mathbf C$ & intersection of cones\\
&$\ex_p$ & exponential map at $p$\\
&$i$ & an automorphism of $T_p\mathcal T$; intersection number\\
&$\ind$ & Morse index of a critical point\\
&$\mathcal T,\mathcal M$ & $\teich$ space, moduli space\\
&$\overline{\mathcal M},\partial\mathcal M,\mathcal S$ & Deligne-Mumford compactification, D-M boundary, stratum\\
&$T^{\sys},T_p^{\sys}\mathcal T$ & major subspace\\
&$T^{\sys\perp},T_p^{\sys\perp}\mathcal T$ & minor subspace\\
&$T^{e}, T^{e\perp\sys}$ & subspace spanned by $\{\nabla l_i\}_{I_{e}}$, orthogonal complement of $T^{e}$ in $T^{\sys}$\\
&$\widetilde\Omega_T,\Omega_T$ & see Definition \ref{setup}\\
&$\widetilde\Phi_T,\Phi_T$ & see Subsection \ref{behaviorinmajor}\\
&$\widetilde\Psi_T,\Psi_T$ & see Definition \ref{ep2}\\
&$\widetilde{H}_T, H_T$ & see Subsection \ref{nondegeneracy}\\

    \end{tabular}
\end{center}
    \label{tab:TableOfNotationForMyResearch}
\end{table*}

\newpage
\bibliographystyle{alpha}
\bibliography{main}

\end{document}